\documentclass[12pt, reqno]{amsart}
\allowdisplaybreaks[3]
\usepackage[left=3cm,right=3cm,top=3cm,bottom=3cm]{geometry}

\usepackage{amsthm}
\usepackage{amscd}
\usepackage{amsmath}
\usepackage{amsfonts}
\usepackage{amssymb}
\usepackage{graphicx}
\usepackage{amsopn}
\usepackage{caption}
\usepackage[usenames]{color}
\usepackage{tikz}
\usetikzlibrary{arrows,shapes,snakes,automata,backgrounds,petri}
\usepackage{textcomp}
\usepackage[all]{xy}

\usepackage[utf8]{inputenc}
\usepackage[T2A]{fontenc}
\usepackage{amsmath,amssymb,amsthm}
\usepackage[english]{babel}
\usepackage{mathtools}

\usepackage{algorithm}
\usepackage{algorithmic}

\usepackage{hyperref}

\theoremstyle{theorem}
\newtheorem{theorem}{Theorem}
\newtheorem{corollary}[theorem]{Corollary}
\newtheorem{proposition}[theorem]{Proposition}
\newtheorem{lemma}[theorem]{Lemma}

\theoremstyle{definition}
\newtheorem{definition}[theorem]{Definition}

\newtheorem{remark}[theorem]{Remark}

\usepackage{verbatim}

\usepackage{subcaption}

\def\RR{\mathcal{R}}
\def\R{\mathbb{R}}
\def\L{\mathcal{L}}

\def\D{\mathcal{D}}

\def\N{\mathbb{N}}

\def\Mat{\mathrm{Mat}}

\def\D{\mathcal{D}}

\def\Rh{\mathfrak{Rh}}
\def\L{\mathfrak{L}}

\numberwithin{theorem}{section}
\allowdisplaybreaks[3]

\begin{document}

\title{Tight Complexity Bounds for Diagram Commutativity Verification}
\thanks{This work is an output of a research project implemented as part of the Basic Research Program at HSE University}
\author{Artem Malko}
\address{A. Malko: \newline National Research University Higher School of Economics, Moscow 101000, Russia}
\email{metraoklam@gmail.com}

\author{Igor Spiridonov}
\address{I. Spiridonov: \newline National Research University Higher School of Economics, Moscow 119048, Russia}
\email{spiridonovia@ya.ru}

\maketitle

\begin{abstract}
    A diagram \(\mathcal{D} = (G, l)\) over a monoid \(M\) is an oriented graph \(G = (V, E)\) endowed with a labeling \(l\colon E \to M\). A diagram is commutative if and only if for any two oriented paths with the same endpoints, the products in \(M\) of their edge labels coincide.
    We propose the first asymptotically optimal algorithm for diagram commutativity verification applicable to all graph families. For graphs with \(\lvert V\rvert \preceq \lvert E\rvert \preceq \lvert V\rvert^2\), which covers most practically relevant cases, our algorithm runs in
    \[
      O\bigl(|V|\,|E|\bigr) \cdot \bigl(T_{\mathrm{equal}} + T_{\mathrm{multi}}\bigr)
    \]
    time; here \(T_{\mathrm{equal}}\) and \(T_{\mathrm{multi}}\) denote the times to perform an equality check and a multiplication in \(M\), respectively.
    We also establish new lower bounds on the numbers of equality checks and multiplications necessary for commutativity verification, which asymptotically match our algorithm’s cost and thus prove its tightness.
\end{abstract}

\section{Introduction}

\subsection{Results overview}
Verifying diagram commutativity is a foundational algorithmic problem in theoretical computer science and combinatorics. The best-known result in this area is the Murota’s algorithm \cite[Theorem 5.2]{murota84}, which applies only to acyclic graphs and requires \(O(|V|^2|E|)\) equality checks. Modern implementations of this algorithm can be found in \cite{kabra20}.

We propose an asymptotically optimal algorithm (see Section \ref{S6}) for diagram commutativity verification that runs in
\begin{equation}\label{eq:complexity}
  \begin{aligned}[c]
    &\hspace{40pt}%
     O\bigl(\min(|V|^2,|E|) \cdot \min(|V|,|E|)+|E|\bigr)\;T_{\mathrm{equal}}+
    \\[-1pt]
    &\hspace{40pt}%
     O\bigl(\min(|V|^2,|E|) \cdot \min(|V|,|E|)\bigr)\;T_{\mathrm{multi}} + O(|V|),
  \end{aligned}
\end{equation}
where \(T_{\mathrm{equal}}\) and \(T_{\mathrm{multi}}\) denote the times to perform an equality check and a multiplication in \(M\), respectively, see Theorem \ref{full_comlexity}.
For the most important graph families -- namely those satisfying \(\lvert V\rvert \preceq \lvert E\rvert \preceq \lvert V\rvert^2\) -- the running time (\ref{eq:complexity}) automatically simplifies to
    \[
    O \left(|V|\,|E|\bigr) \cdot \bigl(T_{\mathrm{equal}} + T_{\mathrm{multi}}\right),
    \]
which strictly improves on the \(O(|V|^2\,|E|)\) bound of all previously known algorithms. Moreover, the exact number of equality checks performed by our algorithm is bounded from above by
\begin{equation}\label{eq:num_checks}
    \min\bigl(|V|^2,|E|\bigr) \cdot \min\bigl(|V|,|E|+1\bigr) + |E|,
\end{equation}
and the exact number of multiplications is bounded from above by
\begin{equation}\label{eq:num_mult}
    \min\bigl(|V|^2,|E|\bigr) \cdot \min\bigl(|V|,|E|+1\bigr),
\end{equation}
see Theorems \ref{algo_eq} and \ref{algo_mult}, respectively.

To complement our upper‐bound analysis, we derive the first exact mathematical lower bounds on the number of elementary operations -- equality checks and multiplications -- required to verify commutativity of a diagram over an arbitrary graph \(G\) with \(\lvert V\rvert\) vertices and \(\lvert E\rvert\) edges.  In particular, we prove that any verification procedure must perform at least
\begin{equation}\label{eq:bound_checks}
    \Omega \left( \min( |V|^2, |E|) \cdot \min(|V|, |E|) + |E| \right)
\end{equation}
equality checks and at least
\begin{equation}\label{eq:bound_mult}
    \Omega \left( \min(|V|^2, |E|) \cdot \min(|V|, |E|) \right)
\end{equation}
multiplications in the worst case, see Theorems \ref{lower_bound} and \ref{lower_bound_mult}, respectively. The hidden constant in both \(\Omega\)-notations can be taken as \(2^{-14}\).

Note that each lower‐bound in \eqref{eq:bound_checks} and \eqref{eq:bound_mult} matches the exact operation counts in \eqref{eq:num_checks} and \eqref{eq:num_mult} up to the universal constant factor. Moreover, by inspecting the bound in \eqref{eq:complexity}, we note that, aside from the equality‐check and multiplication operations, the algorithm performs only
\begin{equation}\label{eq:ram_ops}
O\bigl(\min(|V|^2,|E|) \cdot \min(|V|,|E|)+|E|\bigr) + O(|V|)
\end{equation}
additional basic RAM‐model operations (e.g., indexing, branching, assignment), so that its overall running time obeys the same asymptotic bound. Consequently, these bounds are tight and the algorithm is asymptotically optimal.

In this framework, for an oriented graph \(G = (V, E)\), the minimal number of equality checks is defined as the smallest size of a system of equations between products of edge labels (each product taken over an arbitrary sequence of edges, not necessarily forming a path) whose joint satisfaction guarantees diagram commutativity. We call this the \emph{commutativity rank} \(\eta(G)\) of \(G\), see Definition \ref{def_com_rank}. Similarly, the minimal number of multiplications is the least number of monoid multiplications required to build all left- and right-hand side products appearing in those equations. We call this the \emph{multiplication rank} \(\nu(G)\) of \(G\), see Definition \ref{def_mult_rank}. The algorithm we propose constructs and verifies a family of such equations, proving that they can be constructed in (\ref{eq:ram_ops}) time for any graph \(G = (V, E)\).
Thus, our algorithm provides upper bounds on \(\eta(G)\) and \(\nu(G)\) in terms of \(\lvert V\rvert\) and \(\lvert E\rvert\), see Theorems \ref{upper_bound} and \ref{upper_bound_mult}. 

Crucially, the algorithm only ever multiplies labels along actual paths -- that is, it uses compositions of labels corresponding to consecutive, composable edges. The lower bounds on commutativity and multiplication ranks are proved in the more permissive model that allows arbitrary formal products of edge labels; restricting to a categorical setting, where only composable arrows admit multiplication, can only shrink the space of allowed products, so any commutativity certificate or multiplication scheme there still satisfies the same rank bounds. Hence, the asymptotic lower bounds and the algorithm's optimality carry over immediately to diagrams valued in categories, since the algorithm's operations remain valid and the cost measures can only increase under this restriction.

We work in the full generality of an arbitrary (possibly non-commutative and without inverses) monoid as the target space, thus covering matrix semirings, groups, and most algebraic structure arising in theory or applications. Moreover, our algorithm applies to all oriented graphs -- including those with cycles, multiple edges, loops, and arbitrary connectivity patterns -- without any acyclicity or connectivity assumptions. The only minor technical restriction, used only for the lower bounds, is that $|V|,|E| \geq 4$.

\subsection{Motivation} 
Ever since Eilenberg and Mac Lane’s 1945 landmark ``General Theory of Natural Equivalences''~\cite{eilenberg_maclane45}, commutative diagrams have become a unifying language across mathematics. In the decade that followed, Cartan–Eilenberg’s ``Homological Algebra''~\cite{cartan_eilenberg56} and Gabriel–Zisman’s ``Calculus of Fractions''~\cite{gabriel_zisman67} turned tedious chain-level proofs into systematic arrow-chasing. Grothendieck’s 1957 contributions~\cite{grothendieck57} advanced the art further by defining abelian, fibered, and eventually derived categories entirely through the commutativity of cubes and prisms of morphisms. Lawvere’s 1964 work~\cite{lawvere64} revealed that logical deduction itself is diagrammatic, with quantifiers and implication arising as universal pullbacks and exponentials. By the early 1970s, Scott’s domain-theoretic framework~\cite{scott72} supplied a diagrammatic semantics for the untyped~\(\lambda\)-calculus, while Lambek’s insights~\cite{lambek58} showed that the syntax of natural language can be read from compact-closed string diagrams. Whether one is ensuring that locally defined objects fit together, defining pushouts for programming-language semantics, or studying compositionality in cognitive science, today’s work still comes down to diagram chasing.

As a physical example, in an oriented electrical network with edges labeled by voltage drops (composition as addition), diagram commutativity is equivalent to Kirchhoff’s voltage law -- the total drop around any closed loop is zero.

A further motivation comes from the fact that for any poset \(S\) and any category \(\mathcal{C}\) with finite limits, the category of \(\mathcal{C}\)-valued sheaves on \(S\), endowed with the Alexandrov topology, is equivalent to the category of commutative diagrams in \(\mathcal{C}\) indexed by \(S\)~\cite{curry}. Thanks to this correspondence, verifying diagram commutativity has become an important practical task in computer science, driven by the growing use of sheaf-theoretic frameworks in machine learning and topological data analysis~\cite{ayzenberg25}. Notable examples include heat diffusion on a sheaf~\cite{ghrist19}, sheaf learning~\cite{hancen_ghrist19}, message passing on sheaves~\cite{bodnar20}, sheaf attention~\cite{barbero22}, and simplification of finite sheaved spaces~\cite{malko24}.
     
The problem of diagram commutativity verification has been studied from several perspectives. In combinatorial graph theory, Kainen~\cite{kainen12} introduced the notion of a robust cycle basis to reduce commutativity checks to labels on a specialized cycle basis; Hammack and Kainen~\cite{hammackKainen18} subsequently extended this framework to a broad class of graphs, providing a detailed analysis of the complexity of constructing such bases. In a purely algorithmic setting, Murota~\cite[Theorem 5.2]{murota84} proposed the first algorithm with an explicit asymptotic analysis for acyclic graphs, achieving an \(O(|V|^2|E|)\) bound on equality verifications. Namely, the following result was obtained.
\begin{theorem}\cite[Theorem 5.2]{murota84}\label{murota_thm}
    The commutativity of an acyclic diagram $(G = (V, E), l)$ where $l: E \to F$ can be checked in $O(T_{mult}|V|^2|E|)$ if one multiplication operation in the semigroup $F$ can be done in a constant time $T_{mult}$.
\end{theorem}
Murota’s proof is based on the technique of matroids and uses the notion of a homotopy base~\cite{murotaFujishige87,murota84}. Note that Theorem~\ref{murota_thm} applies only to acyclic diagrams but works over semigroups rather than monoids. Implementations of this algorithm can be found in~\cite{kabra20}.

\subsection{Structure of the paper}
Section \ref{S2} introduces all necessary notation and definitions, and gives formal statements of our main theorems -- both the upper and lower bounds.

Sections \ref{S3}--\ref{S5} develop the lower‐bound side of our results. In Section \ref{S3}, we introduce the key notion of disjoint rhomboids and, for any fixed numbers of vertices and edges, construct graphs containing a sufficiently large collection of disjoint rhomboids; several technical derivations are deferred to Appendix~\ref{appendix}. In Section \ref{S4}, we develop results on the multiplication rank of graphs, which we subsequently employ in our lower‐bound analysis. Section \ref{S5} completes the lower‐bound proof. The key idea is that, in graphs containing a large number of disjoint rhomboids, commutativity cannot be verified with fewer than the asymptotically claimed numbers of multiplications and equality checks.

Sections \ref{S6}--\ref{S7} present the algorithmic side and yield the matching upper bounds. Section \ref{S6} presents the algorithm for verifying diagram commutativity and proves its correctness. In Section \ref{S7}, we compute the exact numbers of multiplications and equality checks executed by our algorithm and derive its asymptotic running time. This simultaneously confirms the algorithm’s optimality and the exactness of our lower‐bound estimates.

\subsection{Acknowledgements}
The authors are grateful to A.~Ayzenberg, M.~Babenko, V.~Gorbounov and F.~Pavutnitskiy for useful comments.

\section{Main results} \label{S2}

\subsection{Main definitions} An \textit{oriented graph} $G$ is a pair $(V, E)$ of finite nonempty sets (called vertices and edges) equipped with two functions $t, o : E \to V$ assigning to every edge its \textit{tail} and \textit{origin}, respectively.

Let $M$ be a monoid, i.e a set equipped with an associative binary operation and an identity element $\mathbf{1}_M \in M$.

\begin{definition}\label{defD}
A diagram $\D = (G, l)$ over a monoid $M$ is a oriented graph $G = (V, E)$ with a labeling $l: E \to M$ on the set of its edges.
\end{definition}

Let $\D = (G, l)$ be a diagram over a monoid $M$ and let $s = (e_1, e_2, \dots, e_k)$ be any sequence of its edges (not necessarily consecutive).
Then we set $$l(s) = l(e_1) l(e_2) \cdots l(e_k) \in M.$$
For the empty sequence $s$ we define $l(s) = \mathbf{1}_M$.

A path in a graph $G = (V, E)$ is a (possibly empty) sequence of edges $p = (e_1, e_2, \dots, e_k)$ equipped with a pair of vertices $o(p), t(p) \in V$ satisfying $t(e_i) = o(e_{i+1})$ for all $1 \leq i \leq k - 1$. If $k > 0$, then we additionally require that $o(p) = o(e_1)$ and $t(p) = t(e_k)$, otherwise the condition $o(p) = t(p)$ should be satisfied. This definition allows us to consider paths of zero length ``located'' at any vertex of $G$. For a path $p = (e_1, e_2, \dots, e_k)$, we denote
$$l(p) = l((e_1, e_2, \dots, e_k)).$$

\begin{definition}\label{defCD}
A diagram $\D = (G, l)$ over $M$ is called \textit{commutative} if for any two paths $p_1, p_2$ satisfying $o(p_1) = o(p_2)$ and $t(p_1) = t(p_2)$ the equality $l(p_1) = l(p_2)$ holds in $M$.
\end{definition}

\begin{definition} \label{system}
    Let $G = (V, E)$ be an oriented graph. A \textit{system of relations} for $G$ is a set of pairs
    $$\RR = \{(s_1, s'_1), (s_2, s'_2), \dots, (s_n, s'_n)\},$$
    where $s_i$ and $s'_i$ are sequences of (not necessarily consecutive) edges of $G$.
    The system of relations $\RR$ is called \textit{complete} if for any monoid $M$ and for any labeling $l: E \to M$ the diagram $D = (G, l)$ is commutative if and only if $l(s_i) = l(s'_i)$ holds for each $1 \leq i \leq n$.
\end{definition}

A complete system of relations $\RR$ is called \textit{minimal} if any of its proper subset $\RR' \subsetneq \RR$ is not a complete system of relations. Note that any minimal complete system of relations does not contain relations of the form $((a), (a))$ and $((), ())$, where $a \in E$ and $()$ denotes the empty edge sequence.

\begin{definition} \label{def_com_rank}
    Let $G$ be an oriented graph. Then \textit{commutativity rank} of $G$
    is defined as the number
    $$\eta(G) = \min_\RR |\RR|,$$
    where the minimum is taken over all complete system of relations for $G$.
\end{definition}

Note that $\eta(G)$ is always achieved on some minimal complete system of relations $\RR$.

The next goal is to define multiplication rank of a graph $\nu(G)$. Roughly speaking, $\nu(G)$ is the minimal amount of multiplications needed to construct a complete system of relations for $G$. In order to give a formal definition, we introduce the following notation. 

Let $S, S'$ be two finite set of edge sequences of $G$. We say that $S'$ can be obtained from $S$ by a multiplication if
$$S' = S \cup \{s \circ s'\},$$
where $s, s' \in S$ and $s \circ s'$ denotes their concatenation.
Denote by $m(S, S') \in \N \cup \{ \infty \}$ the minimal number of multiplications needed to obtain $S'$ from $S$ using multiplications.

\begin{definition} \label{def_mult_m}
    Let $G = (V, E)$ be an oriented graph and let $\RR$ be a system of relations for $G$. Consider the set
    $$S_{\RR} = \bigcup_{(s, s') \in \RR} \{s, s'\}.$$
    We define 
    $$m(\RR) = \min_{S \supseteq S_{\RR}} m(\mathcal{E}, S),$$
    where 
    $$\mathcal{E} = \left( \cup_{e \in E} \{(e)\} \right) \cup \{()\}.$$
\end{definition}

\begin{definition} \label{def_mult_rank}
    Then \textit{multiplication rank} of $G$
    is defined by
    $$\nu(G) = \min_\RR m(\RR),$$
    where the minimum is taken over all complete system of relations for $G$.
\end{definition}

\subsection{Results} In this paper, we prove both lower and upper bounds on the commutativity rank and multiplication rank of graphs. We further establish that these bounds are asymptotically tight: the lower and upper estimates differ by at most a uniform, explicit constant.

Additionally, in Section \ref{S7}, we present an explicit algorithm for diagram commutativity verification. Given a graph \(G\), our algorithm implicitly constructs and verifies a complete system of relations \(\RR\) for \(G\), achieving the upper bounds on the commutativity rank \(\eta(G)\) and the multiplication rank \(\nu(G)\), respectively, see Theorems \ref{algo_eq} and \ref{algo_mult}. 

Our main upper‐bound results are as follows.
\begin{theorem} \label{upper_bound}
For any oriented graph $G = (V, E)$ we have $$\eta(G) \leq \min(|V|^2, |E|) \cdot \min(|V|, |E| + 1) + |E|.$$
\end{theorem}

\begin{theorem} \label{upper_bound_mult}
For any oriented graph $G = (V, E)$ we have $$\nu(G) \leq \min(|V|^2, |E|) \cdot \min(|V|, |E| + 1).$$
\end{theorem}

To prove the lower bounds, we construct a family of graphs exhibiting large commutativity and multiplication ranks. Our main lower‐bound results are as follows.

\begin{theorem} \label{lower_bound}
There exists a uniform constant $C > 0$ (it is enough to take $C = 2^{-14}$) such that for any integers $n, m \geq 4$ there exists a graph $G = (V, E)$ with $|V| = n$ and $|E| = m$, such that
$$\eta(G) \geq C \cdot \left(\min(|V|^2, |E|) \cdot \min(|V|, |E|) + |E|\right).$$
\end{theorem}

\begin{theorem} \label{lower_bound_mult}
There exists a uniform constant $C > 0$ (it is enough to take $C = 2^{-14}$) such that for any integers $n, m \geq 4$ there exists a graph $G = (V, E)$ with $|V| = n$ and $|E| = m$, such that
$$\nu(G) \geq C \cdot \left(\min(|V|^2, |E|) \cdot \min(|V|, |E|)\right).$$
\end{theorem}

Note that our upper bounds are strictly better than all previously known, the same is true for our algorithm complexity.
Since the lower bounds differ by a uniform constant, it follows that our algorithm is asymptotically optimal.

\begin{remark}
In the proofs of Theorems \ref{lower_bound} and \ref{lower_bound_mult}, we fixed the constant to~\mbox{\(C = 2^{-14}\)} purely to simplify the parameter choices. A more careful optimization of the same constructions shows that the same lower bounds remain valid with a significantly larger constant, at the cost of only minor modifications to the numerical estimates.
\end{remark}

\section{Disjoint rhomboids and triploid graphs} \label{S3}

\subsection{Disjoint rhomboids} In this section, we introduce the notion of disjoint rhomboids and prove that for any fixed number of vertices and edges there exist graphs with enough number of disjoint rhomboids. Informally, rhomboid is just a standard (in the sense of diagrams) square inside a graph.
The formal definition is as follows.
    
\begin{definition}
    A $4$-tuple $(a, b, c, d)$ edges of $G$ is said to form a \textit{rhomboid} if $$o(a) = o(c), \; t(b) = t(d), \; t(a) = o(b), \; t(c) = o(d)$$
    and the vertices $o(a), t(a), o(d), t(d)$ are pairwise distinct.
\end{definition}

\begin{figure}[H]
    \centering
    \begin{minipage}{\textwidth}
    \centering
        \begin{tikzpicture}[scale=0.5]
            \filldraw[black] (0,4) circle (3pt);
            \filldraw[black] (-2,0) circle (3pt);
            \filldraw[black] (2,0) circle (3pt);
            \filldraw[black] (0,-4) circle (3pt);

            \draw[thick, ->] (-0.1, 3.8) -- (-1.9, 0.2);
            \draw[thick, ->] (0.1, 3.8) -- (1.9, 0.2);
            \draw[thick, ->] (-1.9, -0.2) -- (-0.15, -3.7);
            \draw[thick, ->] (1.9, -0.2) -- (0.15, -3.7);

            \node at (-1.4, 2.4) {$a$};
            \node at (-1.4, -2.4) {$b$};
            \node at (1.4, 2.4) {$c$};
            \node at (1.4, -2.4) {$d$};
        \end{tikzpicture}
    \end{minipage}%
    \caption{An example of a rhomboid.}
    \label{Rhomboid}
\end{figure}
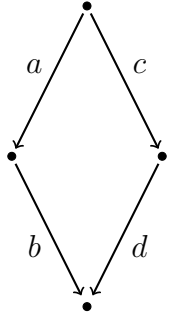

A general rhomboid is shown in Fig. \ref{Rhomboid}. Now we need to define the notion of disjoint rhomboids. Roughly speaking, two rhomboids are said to be disjoint if they do not have a pair of common consecutive oriented edges. The formal definition is as follows.

\begin{definition}
    Two rhomboids $(a, b, c, d)$ and $(a', b', c', d')$ are said to be \textit{disjoint} if the following four conditions are satisfied.

    (1) $a \neq a'$ or $b \neq b'$.

    (2) $a \neq c'$ or $b \neq d'$.

    (3) $c \neq c'$ or $d \neq d'$.

    (4) $c \neq a'$ or $d \neq  b'$.
\end{definition}

\begin{figure}[H]
  \centering
  \begin{subfigure}[b]{0.3\textwidth}
    \centering
    \begin{tikzpicture}[scale=0.5]
      \filldraw[black] (0,0) circle (3pt);
      \filldraw[black] (-3,4) circle (3pt);
      \filldraw[black] (3,4) circle (3pt);
      \filldraw[black] (-3,0) circle (3pt);
      \filldraw[black] (3,0) circle (3pt);
      \filldraw[black] (0,-4) circle (3pt);

      \draw[thin,->] (0, -0.3) -- (0, -3.7);
      \draw[thin,->] (-2.8, -0.3) -- (-0.2, -3.7);
      \draw[thin,->] (2.8, -0.3) -- (0.2, -3.7);
      \draw[thin,->] (-3, 3.7) -- (-3, 0.3);
      \draw[thin,->] (-2.8, 3.7) -- (-0.2, 0.3);
      \draw[thin,->] (2.8, 3.7) -- (0.2, 0.3);
      \draw[thin,->] (3, 3.7) -- (3, 0.3);

      \draw[thick,red]   (-0.1, -0.3) -- (-0.1, -3.7);
      \draw[thick,red]   (-2.95, -0.3) -- (-0.35, -3.7);
      \draw[thick,red]   (-3.1,  3.7) -- (-3.1,  0.5);
      \draw[thick,red]   (-2.95, 3.7) -- (-0.35,  0.3);

      \draw[thick,blue]  (0.1,  -0.3) -- (0.1,  -3.7);
      \draw[thick,blue]  (2.95, -0.3) -- (0.35, -3.7);
      \draw[thick,blue]  (3.1,   3.7) -- (3.1,   0.5);
      \draw[thick,blue]  (2.95,  3.7) -- (0.35,   0.3);
    \end{tikzpicture}
    \subcaption{}\label{fig:2a}
  \end{subfigure}\hfill
  \begin{subfigure}[b]{0.3\textwidth}
    \centering
    \begin{tikzpicture}[scale=0.5]
      \filldraw[black] (-2,4) circle (3pt);
      \filldraw[black] (2,4)  circle (3pt);
      \filldraw[black] (-2,0) circle (3pt);
      \filldraw[black] (2,0)  circle (3pt);
      \filldraw[black] (0,-4) circle (3pt);

      \draw[thin,->] (-2, 3.7) -- (-2, 0.3);
      \draw[thin,->] (2,  3.7) -- (2,  0.3);
      \draw[thin,->] (-1.7,3.7) -- (1.7, 0.3);
      \draw[thin,->] (1.7, 3.7) -- (-1.7,0.3);
      \draw[thin,->] (-1.85,-0.3) -- (-0.15,-3.7);
      \draw[thin,->] (1.85,-0.3) -- (0.15,-3.7);

      \draw[thick,red]  (-2.1,  3.7) -- (-2.1, 0.5);
      \draw[thick,red]  (-1.85, 3.7) -- (1.55, 0.3);
      \draw[thick,red]  (-2,   -0.3) -- (-0.3, -3.7);
      \draw[thick,red]  (1.7,  -0.3) -- (0,    -3.7);

      \draw[thick,blue] (2.1,   3.7) -- (2.1,   0.5);
      \draw[thick,blue] (1.85,  3.7) -- (-1.55, 0.3);
      \draw[thick,blue] (2,    -0.3) -- (0.3,  -3.7);
      \draw[thick,blue] (-1.7, -0.3) -- (0,    -3.7);
    \end{tikzpicture}
    \subcaption{}\label{fig:2b}
  \end{subfigure}\hfill
  \begin{subfigure}[b]{0.3\textwidth}
    \centering
    \begin{tikzpicture}[scale=0.5]
      \filldraw[black] (0,0) circle (3pt);
      \filldraw[black] (0,4) circle (3pt);
      \filldraw[black] (-3,0) circle (3pt);
      \filldraw[black] (3,0) circle (3pt);
      \filldraw[black] (0,-4) circle (3pt);

      \draw[thin,->] (0,  -0.3) -- (0,   -3.7);
      \draw[thin,->] (-2.8,-0.3) -- (-0.2, -3.7);
      \draw[thin,->] (2.8,-0.3) -- (0.2,  -3.7);
      \draw[thin,->] (-0.2,3.7) -- (-2.8,  0.3);
      \draw[thin,->] (0.2, 3.7) -- (2.8,   0.3);
      \draw[thin,->] (0,   3.7) -- (0,     0.3);

      \draw[thick,red]  (-0.1, -0.3) -- (-0.1, -3.7);
      \draw[thick,red]  (-2.95,-0.3) -- (-0.35,-3.7);
      \draw[thick,red]  (-0.35,3.7) -- (-2.95, 0.3);
      \draw[thick,red]  (-0.1, 3.7) -- (-0.1,   0.3);

      \draw[thick,blue] (0.1,  -0.3) -- (0.1,   -3.7);
      \draw[thick,blue] (2.95, -0.3) -- (0.35, -3.7);
      \draw[thick,blue] (0.35,  3.7) -- (2.95,  0.3);
      \draw[thick,blue] (0.1,   3.7) -- (0.1,    0.3);
    \end{tikzpicture}
    \subcaption{}\label{fig:2c}
  \end{subfigure}

  \caption{On Figures \subref{fig:2a} and \subref{fig:2b} the pairs of rhomboids are disjoint, on \subref{fig:2c} – not disjoint.}
  \label{fig:2}
\end{figure}
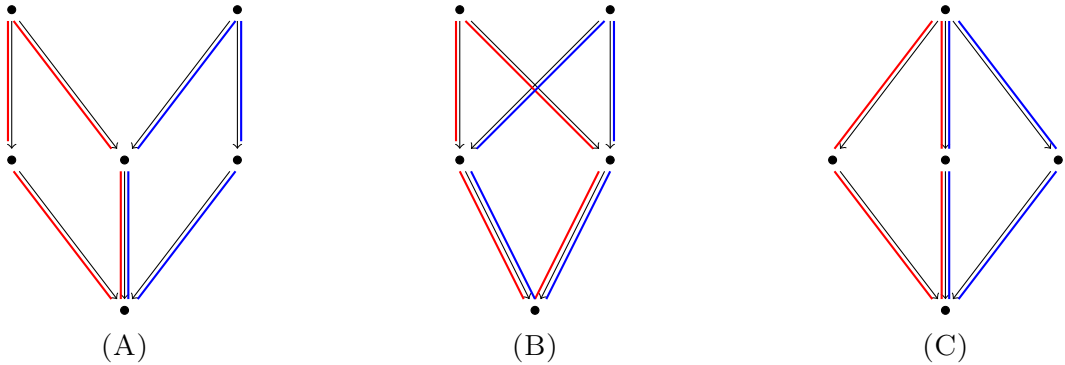

Figures \ref{fig:2a} and \ref{fig:2b} provide examples of disjoint pairs of rhomboids; on Figure \ref{fig:2c} one can find an example of two rhomboids which are not disjoint. The key definition is as follows.

\begin{definition}
Let $G$ be a graph. Then $\Rh(G)$ is defined as 
$$\Rh(G) = \max_{Rh}|Rh|$$
where maximum is taken over all families $Rh$ of pairwise disjointed rhomboids contained in $G$.
\end{definition}

The key idea is that, if $\Rh(G)$ is large enough, then $\eta(G)$ and $\nu(G)$ should also be large, if some additional conditions are satisfied. Roughly speaking, the reason is, under these conditions, each rhomboid from a family of disjoint ones ``requires'' an additional commutativity check.

\subsection{Triploids} Let us provide an explicit construction of a series of graphs with $\Rh(G)$ large enough.
We call these graphs \textit{triploids}.
This construction is based on the following definition.

\begin{figure}[H]
    \centering
    \begin{minipage}{\textwidth}
    \centering
        \begin{tikzpicture}[scale=0.5]
            \filldraw[black] (-8,4) circle (3pt);
            \filldraw[black] (-5,4) circle (3pt);
            \filldraw[black] (1,4) circle (3pt);
            \filldraw[black] (4,4) circle (3pt);
            
            \filldraw[black] (-2.5,4) circle (1pt);
            \filldraw[black] (-2,4) circle (1pt);
            \filldraw[black] (-1.5,4) circle (1pt);
            \filldraw[black] (-2.5,0) circle (1pt);
            \filldraw[black] (-2,0) circle (1pt);
            \filldraw[black] (-1.5,0) circle (1pt);
            \filldraw[black] (-2.5,-4) circle (1pt);
            \filldraw[black] (-2,-4) circle (1pt);
            \filldraw[black] (-1.5,-4) circle (1pt);

            \filldraw[black] (8,1) circle (3pt);
            \filldraw[black] (10,1) circle (3pt);
            \filldraw[black] (9,-1) circle (3pt);
            
            \filldraw[black] (8.7,0.4) circle (1pt);
            \filldraw[black] (9.3,0.4) circle (1pt);
            \filldraw[black] (9,-0.2) circle (1pt);
            \filldraw[black] (9,0.1) circle (1pt);

            \filldraw[black] (0,0) circle (3pt);
            \filldraw[black] (-4,0) circle (3pt);

            \filldraw[black] (-8,-4) circle (3pt);
            \filldraw[black] (-5,-4) circle (3pt);
            \filldraw[black] (1,-4) circle (3pt);
            \filldraw[black] (4,-4) circle (3pt);

            \draw[thin, ->] (-7.8, 3.8) -- (-4.2, 0.2);
            \draw[thin, ->] (-7.6, 3.8) -- (-0.2, 0.2);
            \draw[thin, ->] (-4.9, 3.8) -- (-4.1, 0.2);
            \draw[thin, ->] (-4.8, 3.8) -- (-0.1, 0.2);
            \draw[thin, ->] (0.8, 3.8) -- (-3.9, 0.2);
            \draw[thin, ->] (0.9, 3.8) -- (0.1, 0.2);
            \draw[thin, ->] (3.6, 3.8) -- (-3.8, 0.2);
            \draw[thin, ->] (3.8, 3.8) -- (0.2, 0.2);

            \draw[thin, <-] (-7.8, -3.8) -- (-4.2, -0.2);
            \draw[thin, <-] (-7.6, -3.8) -- (-0.2, -0.2);
            \draw[thin, <-] (-4.9, -3.8) -- (-4.1, -0.2);
            \draw[thin, <-] (-4.8, -3.8) -- (-0.1, -0.2);
            \draw[thin, <-] (0.8, -3.8) -- (-3.9, -0.2);
            \draw[thin, <-] (0.9, -3.8) -- (0.1, -0.2);
            \draw[thin, <-] (3.6, -3.8) -- (-3.8, -0.2);
            \draw[thin, <-] (3.8, -3.8) -- (0.2, -0.2);

            \filldraw[black] (-8,2.3 + 3) circle (1pt);
            \filldraw[black] (-8,2.5 + 3) circle (1pt);
            \filldraw[black] (-8,2.7 + 3) circle (1pt);

            \draw (-7.9,4.1) to [out=150,in=180,looseness=1] (-8,3 + 3);
            \draw[->] (-8,3 + 3) to [out=0,in=30,looseness=1] (-8.1,1.1 + 3);
            \draw (-7.9,4.1) to [out=150,in=180,looseness=1] (-8,3.5 + 3);
            \draw[->] (-8,3.5 + 3) to [out=0,in=30,looseness=1] (-8.1, 1.1+ 3);
            \draw (-7.9,4.1) to [out=150,in=180,looseness=1] (-8, 2 + 3);
            \draw[->] (-8,2 + 3) to [out=0,in=30,looseness=1] (-8.1, 1.1+ 3);

            \node at (-10, 4) {$V^1$};
            \node at (-6, 0) {$V^2$};
            \node at (-10, -4) {$V^3$};
            \node at (7, 0) {$V^0$};
            \node at (3, 2) {$E^{1, 2}$};
            \node at (3, -1.5) {$E^{2, 3}$};
            \node at (-7, 3+ 3) {$L$};
        \end{tikzpicture}
    \end{minipage}%
    \caption{\textit{Triploid visualization}.}
    \label{fig:gen_triploid}
\end{figure}
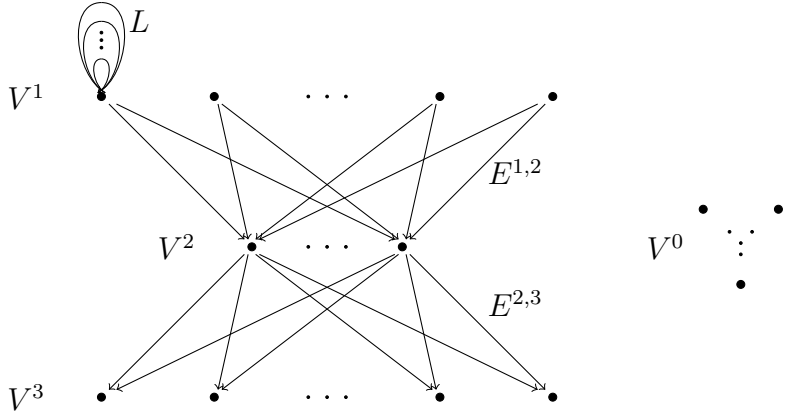

\begin{definition} \label{def_rhomb}
    Let $n_0, n_1, n_2, n_3, e$ be non-negative integers such that $e \geq n_2 (n_1 + n_3)$ and $n_1 > 0$. We define the \textit{triploid} $T(n_1, n_2, n_3, n_0, e)$ as the following graph $G = (V, E)$.

    $\bullet$ $V =  V^0 \sqcup V^1 \sqcup V^2 \sqcup V^3 $ 

    $\bullet$ $V^i = \{ v_1^i, v_2^i, \dots, v_{n_i}^i \}, \ i = 0, 1, 2, 3$

    $\bullet$ $E = E^{1, 2} \sqcup E^{2, 3} \sqcup L$

    $\bullet$ $E^{i, j} = \{ e^{i, j}_{k, l} \ | \ k \in \{ 1, \dots, n_i \}, l \in \{ 1, \dots, n_j \} \}$, $(i, j) = (1, 2), (2, 3)$

    $\bullet$ $t(e^{i, j}_{k, l}) = v^j_l$, $o(e^{i, j}_{k, l}) = v^i_k$

    $\bullet$ $ L = \{ l_1, \dots, l_{e - n_2(n_1 + n_3)} \}$, $o(l_i) = t(l_i) = v_1^1$
\end{definition}

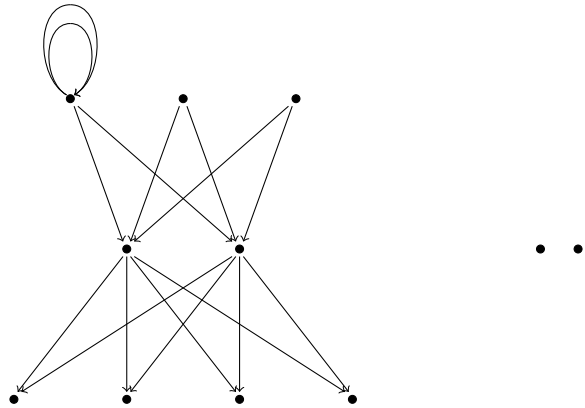
\begin{figure}[H]
    \centering
    \begin{minipage}{\textwidth}
    \centering
        \begin{tikzpicture}[scale=0.5]
            \filldraw[black] (0, 0) circle (3pt);
            \filldraw[black] (-3, 0) circle (3pt);

            \filldraw[black] (-1.5, 4) circle (3pt);
            \filldraw[black] (-4.5, 4) circle (3pt);
            \filldraw[black] (1.5, 4) circle (3pt);

            \filldraw[black] (0, -4) circle (3pt);
            \filldraw[black] (-3, -4) circle (3pt);
            \filldraw[black] (-6, -4) circle (3pt);
            \filldraw[black] (3, -4) circle (3pt);

            \draw[thin, ->] (-4.4, 3.8) -- (-3.1, 0.2);
            \draw[thin, ->] (-4.3, 3.8) -- (-0.2, 0.2);
            \draw[thin, ->] (-1.4, 3.8) -- (-0.1, 0.2);
            \draw[thin, ->] (-1.6, 3.8) -- (-2.9, 0.2);
            \draw[thin, ->] (1.4, 3.8) -- (0.1, 0.2);
            \draw[thin, ->] (1.3, 3.8) -- (-2.8, 0.2);

            \draw[thin, ->] (0, -0.2) -- (0, -3.8);
            \draw[thin, ->] (-3, -0.2) -- (-3, -3.8);
            
            \draw[thin, ->] (-0.1, -0.2) -- (-2.9, -3.8);
            \draw[thin, ->] (-0.2, -0.2) -- (-5.8, -3.8);
            \draw[thin, ->] (0.1, -0.2) -- (2.9, -3.8);
            
            \draw[thin, ->] (-3.1, -0.2) -- (-5.9, -3.8);
            \draw[thin, ->] (-2.9, -0.2) -- (-0.1, -3.8);
            \draw[thin, ->] (-2.8, -0.2) -- (2.8, -3.8);

            \filldraw[black] (8,0) circle (3pt);
            \filldraw[black] (9,0) circle (3pt);
            \draw (-4.6, 4.1) to [out=150,in=180,looseness=1] (-4.5 , 6);
            \draw[->] (-4.5 , 6) to [out=0,in=30,looseness=1] (-4.4, 4.1);
            \draw (-4.6, 4.1) to [out=150,in=180,looseness=1] (-4.5 , 6.5);
            \draw[->] (-4.5 , 6.5) to [out=0,in=30,looseness=1] (-4.4, 4.1);
        \end{tikzpicture}
    \end{minipage}%
    \caption{\textit{Triploid $T(3, 2, 4, 2, 16)$.}}
    \label{fig:triploid}
\end{figure}

Informally, a general triploid $T(n_1, n_2, n_3, n_0, e)$ consists of three rows of vertices $V^1, V^2, V^3$ and a separate group of vertices $V^0$, where $|V^i| = n_i$. There is an oriented edge between each pair of vertices 
from $V^1$ and $V^2$, as well as each pair from $V^2$ and $V^3$. Other edges are loops with endpoints at the first 
vertex of $V^1$, and there are $e$ edges in total.
Figure \ref{fig:gen_triploid} provides a general visualization of a triploid; on Figure \ref{fig:triploid} one can find the triploid $T(3, 2, 4, 2, 16)$.

Denote by $\L(G)$ the number of loops in $G$. 
The following lemma provides a lower bound for the amount of disjoint rhomboids in a general triploid.

\begin{lemma} \label{triploid_rhombs}
    Let $G$ be a triploid $T(n_1, n_2, n_3, n_0, e)$. Then $\Rh(G) \geq n_1 \cdot n_3 \cdot \lfloor \frac{n_2}{2} \rfloor$ and $\L(G) = e - n_2 \cdot (n_1 + n_3)$.
\end{lemma}

\begin{proof}
    Since there are no multiple edges in triploids (except loops), any rhomboid is uniquely determined by its vertices.
    Consider the following family of rhomboids.
    $$Rh = \{(v^1_i, v^2_{2j-1}, v^3_k, v^2_{2j}) \; | \; 1 \leq i \leq n_1, \; 1 \leq j \leq \lfloor \tfrac{n_2}{2} \rfloor, \; 1 \leq k \leq n_3 \}, $$
    $$|Rh| = n_1 \cdot n_3 \cdot \lfloor \tfrac{n_2}{2} \rfloor.$$

    One can easily check that any two distinct rhomboids from $Rh$ are disjoint. Indeed, otherwise they would have a common vertex from each of $V^1, V^2, V^3$ simultaneously; this would contradict the construction of $Rh$.

    The claim about the number of loops immediately follows from Definition \ref{def_rhomb}, since the edges of $T(n_1, n_2, n_3, n_0, e)$ forming loops are precisely $L = \{ l_1, \dots, l_{e - n_2(n_1 + n_3)} \}$, so 
    $$\L(G) = |L| = e - n_2 \cdot (n_1 + n_3).$$
\end{proof}

The central result of this section is as follows.

\begin{theorem} \label{nu_ge}
    There exists a uniform constant $C > 0$ (it is sufficient to take $C = 2^{-14}$) such that for any integers $n, m \geq 4$ there exists an oriented graph $G = (V, E)$ with $|V| = n$ and $|E| = m$ such that 
    \[
    \Rh(G) + \L(G) \geq C \cdot \left(\min(m, n^2) \cdot \min(m, n) + m\right)
    \]
    and
    \[
    \Rh(G) \geq C \cdot \left(\min(m, n^2) \cdot \min(m, n)\right).
    \]
    Moreover, the graph $G$ can be chosen to be isomorphic to a triploid $T(n_1, n_2, n_3, n_0, e)$ for some non-negative integers $n_1, n_2, n_3, n_0, e$.
\end{theorem}

The proof of Theorem \ref{nu_ge} is deferred to Appendix \ref{appendix} due to its technical complexity.

\section{Graph multiplication rank} \label{S4}

\subsection{Monotonicity}
In this section we prove some properties of graph multiplication rank. We begin with a series of results on the monotonicity properties of the $m$ function, introduced in Definition \ref{def_mult_m}. The first observation is as follows.

\begin{lemma} \label{m_momotonous}
    Let $\RR$ and $\RR'$ be systems of relations for an oriented graph $G$ such that $\RR' \subseteq \RR$. Then $m(\RR') \leq m(\RR)$.
\end{lemma}

\begin{proof}
    This is straightforward since $S_{\RR'} \subseteq S_{\RR}$ and hence
    $$m(S_{\RR'}) = \min_{S \supseteq S_{\RR'}} m(\mathcal{E}, S) \leq \min_{S \supseteq S_{\RR}} m(\mathcal{E}, S) = m(S_{\RR}).$$
\end{proof}

Lemma \ref{m_momotonous} automatically implies the following useful fact.

\begin{corollary} \label{minimal_nu}
    Let $G$ be an oriented graph. Then $\nu(G) = m(\RR)$ for some minimal complete system of relations $\RR$ for $G$.
\end{corollary}

We also need the following result.

\begin{lemma} \label{nu_is_monotonous}
    Let $G$ be a graph and $\RR$ be any system of relations for $G$. 
    Let $\RR'$ be a system of relations obtained from $\RR$ by removing all loop-edges from all edge sequences. Then $m(\RR') \leq m(\RR)$.
\end{lemma}

\begin{proof}
    Let
    $$\mathcal{E} = S_0 \subseteq S_1 \subseteq S_2 \subseteq \dots \subseteq S_{m(\RR)} \supseteq S_{\RR}$$
    is a sequence of sets realizing $m(\RR)$, i.e. $S_{i}$ can be obtained by multiplication from $S_{i-1}$ for all $i \in \{1, 2, \dots, m(\RR)\}$.
    Denote by $S'_i$ the set obtained from $S_i$ by removing all loop-edges from all edge sequences of $S_i$. Then it is straightforward to check that the sequence 
    $$\mathcal{E} = S_0 \subseteq S'_1 \subseteq S'_2 \subseteq \dots \subseteq S'_{m(\RR)} \supseteq S_{\RR'}$$
    satisfies the same property. Therefore, $m(\RR') \leq m(\RR)$.
\end{proof}

\subsection{Some lower bounds}
The next goal is to prove a lower bound for the $m$ function in some special case.

Let us introduce some additional notation. Let $\RR$ be a complete system of relations for an oriented graph $G$. We denote by $\RR_B \subseteq \RR$ the subset consisting of all relations $(s, s') \in \RR$ such that $|s| \geq 2$ and $|s'| \geq 2$. Consider the set
    $$S_{\RR_B} = \bigcup_{(s, s') \in \RR_B} \{s, s'\}.$$
First, we prove the following.

\begin{lemma} \label{m_bound}
    Let $\RR$ be a complete system of relations for an oriented graph $G$.
    In the above notation we have $m(\RR) \geq |S_{\RR_B}|$.
\end{lemma}

\begin{proof}
    Let
    $$\mathcal{E} = S_0 \subseteq S_1 \subseteq S_2 \subseteq \dots \subseteq S_{m(\RR)} \supseteq S_{\RR}$$
    is a sequence of sets realizing $m(\RR)$, i.e. $S_{i}$ can be obtained by multiplication from $S_{i-1}$ for all $i \in \{1, 2, \dots, m(\RR)\}$.
    
    Note that $S_{\RR_B} \cap \mathcal{E} = \varnothing$ and $\mathcal{E}, S_{\RR_B} \subseteq S_{m(\RR)}$, so 
    $$|S_{m(\RR)}| \geq |\mathcal{E}| + |S_{\RR_B}|.$$ 
    On the other hand, $|S_i| \leq |S_{i-1}| + 1$ for all $i \in \{1, 2, \dots, m(\RR)\}$, hence 
    $$|S_{m(\RR)}| \leq |\mathcal{E}| + m(\RR).$$
    Combining these two inequalities, we obtain that $m(\RR) \geq |S_{\RR_B}|$.
\end{proof}

The next step is the following lemma.

\begin{lemma} \label{complete_ineq}
    Let $\RR$ be a minimal complete system of relations for an oriented graph $G$.
    In the above notation we have $m(\RR) \geq |\RR_B|$.
\end{lemma}

\begin{proof}
    In view of Lemma \ref{m_bound}, it suffices to check that $|S_{\RR_B}| \geq |\RR_B|$. We prove by contradiction. Assume the converse, i.e. the inequality $|S_{\RR_B}| < |\RR_B|$ holds.

    Consider an unoriented graph $\Gamma = (V_\Gamma, E_\Gamma)$ with the vertex set $V_\Gamma = S_{\RR_B}$, and we add an edge $(s, s')$ to $\Gamma$ for each $(s, s') \in \RR_B$, so $|E_\Gamma| = |\RR_B|$.
    
    By assumption, $|E_\Gamma| = |\RR_B| > |S_{\RR_B}| = |V_\Gamma|$, so $\Gamma$ contains a cycle. Let $(s, s') \in \RR_B$ corresponds to any edge of this cycle. Then for any labeling $l: G \to M$ the equality $l(s) = l(s')$ holds automatically by transitivity if all the relations from $\RR_B \setminus \{(s, s')\}$ holds. Therefore, since $\RR_B \subseteq \RR$, we obtain that $\RR$ is not minimal. This contradiction proves the result.
\end{proof}

Lemma \ref{complete_ineq} automatically implies the following corollary, which is the most important result of this section.

\begin{corollary} \label{rb_is_r}
    Let $\RR$ be a minimal complete system of relations for an oriented graph $G$. Assume that for any relation $(e, e') \in \RR$ we have $|e| \geq 2$ and $|e'| \geq 2$.
    Then $m(\RR) \geq |\RR|$.
\end{corollary}

\section{Lower bounds for commutativity and multiplication rank} \label{S5}
\subsection{Sketch of proof}
In this section we prove lower bounds for commutativity and multiplication rank, i.e. provide proofs of Theorems \ref{lower_bound} and \ref{lower_bound_mult}. These proofs are divided into several steps, so we start with a sketch.

First, let us introduce some notation.
Recall that edges $e, e' \in E$ are called multiple if 
$$t(e) = t(e') \neq o(e) = o(e'),$$
i.e. we additionally require them to be not loops.
We say that edges $a, b, c \in E$ form a triangle if 
$$o(a) = o(b), \; t(b) = o(c), \; t(c) = t(a),$$
and neither of them is a loop.

Second, we call a graph \textit{quasi-acyclic} if it has no directed cycles that visit two or more distinct vertices (self-loops permitted); equivalently, every strongly connected component is a singleton.

Finally, we say that a graph $G$ is $k$-\textit{path-bounded} if every oriented path avoiding loops edges has length at most $k$ (we will use only $k=2$).

The key observation is as follows.

\begin{remark} \label{remark_triploid}
Let $G$ be a triploid. Then $G$ is quasi-acyclic, $2$-path-bounded, and does not have multiples edges and triangles.
\end{remark}

The main results of this section are the following two theorems.

\begin{theorem}\label{eq_amount_theorem}
Let $G$ be a quasi-acyclic $2$-path-bounded graph without multiple edges and triangles. Then $\eta(G) \geq \Rh(G) + \L(G)$. 
\end{theorem}

\begin{theorem}\label{mult_amount_theorem}
    Let $G$ be a quasi-acyclic $2$-path-bounded graph without multiple edges and triangles. Then $\nu(G) \geq \Rh(G)$.
\end{theorem}

Let us now deduce Theorems \ref{lower_bound} and \ref{lower_bound_mult} from Theorems \ref{eq_amount_theorem} and \ref{mult_amount_theorem}.

\begin{proof}[Proof of Theorem \ref{lower_bound}]
    Theorem \ref{nu_ge} implies that for any integers $n, m \geq 4$ there exists a graph $G = (V, E)$ with $|V| = n$ and $|E| = m$, such that the inequality 
     \[
    \Rh(G) + \L(G) \geq 2^{-14} \cdot \left(\min(m, n^2) \cdot \min(m, n) + m\right)
    \]
    holds and such that $G \cong T(n_1, n_2, n_3, n_0, e)$ for some integers $n_1, n_2, n_3, n_0, e$. By Remark \ref{remark_triploid}, the graph $G$ satisfies the conditions of Theorem \ref{eq_amount_theorem}, hence we have $\eta(G) \geq \Rh(G) + \L(G)$. Therefore, we obtain
    \[
    \eta(G) \geq \Rh(G) + \L(G) \geq 2^{-14} \cdot \left(\min(m, n^2) \cdot \min(m, n) + m\right).
    \]
\end{proof}

\begin{proof}[Proof of Theorem \ref{lower_bound_mult}]
    Theorem \ref{nu_ge} implies that for any integers $n, m \geq 4$ there exists a graph $G = (V, E)$ with $|V| = n$ and $|E| = m$, such that the inequality 
    \[
    \Rh(G) \geq C \cdot \left(\min(m, n^2) \cdot \min(m, n)\right).
    \]
    holds and such that $G \cong T(n_1, n_2, n_3, n_0, e)$ for some integers $n_1, n_2, n_3, n_0, e$. By Remark \ref{remark_triploid}, the graph $G$ satisfies the conditions of Theorem \ref{mult_amount_theorem}, hence we have $\nu(G) \geq \Rh(G)$. Therefore, we obtain
    \[
    \nu(G) \geq \Rh(G) \geq 2^{-14} \cdot \left(\min(m, n^2) \cdot \min(m, n)\right).
    \]
\end{proof}

Our strategy is first to prove Theorems \ref{eq_amount_theorem} and \ref{mult_amount_theorem} for loopless quasi-acyclic graphs (i.e., acyclic), and then extend to all quasi-acyclic graphs.

\subsection{Loopless case} Our goal is to prove the following two propositions.

\begin{proposition} \label{no_loops}
Let $G = (V, E)$ be a quasi-acyclic $2$-path-bounded graph without loops, multiple edges and triangles. Then $\eta(G) \geq \Rh(G)$. 
\end{proposition}

\begin{proposition} \label{no_loops_mult}
Let $G = (V, E)$ be a quasi-acyclic $2$-path-bounded graph without loops, multiple edges and triangles. Then $\nu(G) \geq \Rh(G)$. 
\end{proposition}

Before we prove Propositions \ref{no_loops} and \ref{no_loops_mult} we need several lemmas. First, let us make the following observation, claiming that for the graphs satisfying conditions of Propositions \ref{no_loops} and \ref{no_loops_mult} it is enough to verify commutativity only for paths of length $2$.

\begin{lemma} \label{length_2}
    Let $G = (V, E)$ be a quasi-acyclic $2$-path-bounded graph without loops, multiple edges and triangles.
    Then a diagram $D = (G, l)$ is commutative if and only if for any two paths $p_1, p_2$ in $G$ with the same endpoints and satisfying $|p_1| = |p_2| = 2$ we have $l(p_1) = l(p_2)$. 
\end{lemma}

\begin{proof}
    It is enough to check that if for any two paths $p_1, p_2$ in $G$ with the same endpoints and satisfying $|p_1| = |p_2| = 2$ we have $l(p_1) = l(p_2)$, then $D$ is commutative.
    
    Indeed, let $p, p'$ be any pair of paths in $G$. Since $G$ has no loops and  $2$-path-bounded, then $|p|, |p'| \leq 2$. If $|p| = |p'| = 2$, then $l(p) = l(p')$ by assumption. The case $|p| = 1$, $|p'| = 2$ or $|p| = 2$, $|p'| = 1$ is impossible since $G$ does not have triangles. If $|p| = |p'| = 1$, then $p = p'$ since $G$ has no multiple edges, so $l(p) = l(p')$ holds automatically.

    The case $|p| = 0$, $|p|' \neq 0$ or $|p| \neq 0$, $|p'| = 0$ is impossible since $G$ is quasi-acyclic and has no loops (i.e., acyclic). Finally, if $|p| = |p'| = 0$, then by definition we have $l(p) = \mathbf{1}_M = l(p')$.
\end{proof}

Let us introduce some notation.
Here and throughout, when we say that a diagram takes values in a ring $R$, we implicitly view $R$ as a monoid under multiplication. A diagram $D = (G, l)$ with values in a ring $R$ is called $3$-\textit{vanishing} if for any (not necessarily consecutive) three edges $e_1, e_2, e_3 \in E$ we have $l(e_1)l(e_2)l(e_3) = \mathbf{0}_R$. 

By $\mathrm{Mat}_k(\R)$ we denote the ring of $(k \times k)$-matrices over $\mathbb{R}$. For $i, j \in \{1, \dots, k\}$, we denote by $E_{i, j} \in \mathrm{Mat}_k(\R)$ the standard matrix unit.

\begin{lemma} \label{NZedge}
    Let $G = (V, E)$ be a quasi-acyclic $2$-path-bounded graph without loops, multiple edges and triangles.
    Then for any edge $e \in E$ there exists a $3$-vanishing commutative diagram $D = (G, l)$ with values in a ring $R$ such that $l(e) \neq \mathbf{0}_R$ and $l(e') = \mathbf{0}_R$ for any other edge $e' \neq e$.
\end{lemma}

\begin{proof}
    We take $R = \mathrm{Mat}_2(\R)$ and consider the following labeling $l : E \to \mathrm{Mat}_2(\R)$

    $$l(x) = \begin{cases}
        E_{1, 2} & \text{if $x = e$}\\
        \mathbf{0}_R & \text{otherwise.}
    \end{cases}$$
    
    The diagram $\D = (G, l)$ is $3$-vanishing since for any $r_1, r_2, r_3 \in \{E_{1, 2}, \mathbf{0}_R\}$ we have $r_1 \cdot r_2 \cdot r_3 = \mathbf{0}_R$.
    
    Let us prove commutativity of $D = (G, l)$. By Lemma \ref{length_2}, it is enough to check that for any paths $p_1, p_2$ in $G$ with same endpoints and satisfying $|p_1| = |p_2| = 2$ we have $l(p_1) = l(p_2)$. This is straightforward since for any $r_1, r_2 \in \{E_{1, 2}, \mathbf{0}_R\}$ we have $r_1 \cdot r_2 = \mathbf{0}_R$.
\end{proof}

\begin{lemma} \label{eneqf}
   Let $G = (V, E)$ be a quasi-acyclic $2$-path-bounded graph without loops, multiple edges and triangles.
   Then for any distinct edges $e, e' \in E$ there exists a $3$-vanishing commutative diagram $D = (G, l)$ with values in a ring $R$ such that $l(e) \neq l(e')$.
\end{lemma}

\begin{proof}
     Immediately follows from Lemma \ref{NZedge}.
\end{proof}

\begin{lemma} \label{NZef}
    Let $G = (V, E)$ be a quasi-acyclic $2$-path-bounded graph without loops, multiple edges and triangles. Then for any distinct edges $e, e' \in E$ there exists a $3$-vanishing commutative diagram $D = (G, l)$ with values in a ring $R$ such that $l(e)l(e') \neq \mathbf{0}_R$.
\end{lemma}

\begin{proof}
    \textbf{Case 1.} $t(e) = o(e')$.

    We take $R = \mathrm{Mat}_3(\R)$ and consider the following labeling $l : E \to \mathrm{Mat}_3(\R)$

    $$l(x) = \begin{cases}
        E_{1, 2} & \text{if $o(x) = o(e)$}\\
        E_{2, 3} & \text{if $t(x) = t(e')$}\\
        \mathbf{0}_R & \text{otherwise.}
    \end{cases}
    $$

    Note that $l$ is well-defined. Indeed, since $G$ does not contain triangles, it follows that there is no edge $x \in E$ with $o(x) = o(e)$ and $t(x) = t(e')$. We also note that $o(e) \neq t(e')$ since $G$ is quasi-acyclic.
    The diagram $\D = (G, l)$ is $3$-vanishing since for any $r_1, r_2, r_3 \in \{E_{1, 2}, E_{2, 3}, \mathbf{0}_R\}$ we have $r_1 \cdot r_2 \cdot r_3 = \mathbf{0}_R$.
    
    Let us prove commutativity of $D = (G, l)$.
    By Lemma \ref{length_2}, it is enough to check that for any paths $p_1, p_2$ in $G$ with same endpoints and satisfying $|p_1| = |p_2| = 2$ we have $l(p_1) = l(p_2)$.  If $o(p_1) = o(e)$ and $t(p_1) = t(e')$, by construction of $l$ we have that $l(p_1) = E_{1, 2} E_{2, 3} = l(p_2)$. Otherwise, since $o(e) \neq t(e')$, both $p_1$ and $p_2$ contain at least one edge $f$ such that $o(f) \neq o(e)$ and $t(f) \neq t(e')$. By definition of $l$, it follows that $l(f) = \mathbf{0}_R$, so $l(p_1) = \mathbf{0}_R = l(p_2)$.

    \textbf{Case 2.} $t(e) \neq o(e')$.

     We take $R = \mathrm{Mat}_3(\R)$ and consider the following labeling $l : E \to \mathrm{Mat}_3(\R)$

    $$l(x) = \begin{cases}
        E_{1, 2} & \text{if $x = e$}\\
        E_{2, 3} & \text{if $x = e'$}\\
        \mathbf{0}_R & \text{otherwise.}
    \end{cases}
    $$
    Again, the diagram $\D = (G, l)$ is $3$-vanishing since for any $r_1, r_2, r_3 \in \{E_{1, 2}, E_{2, 3}, \mathbf{0}_R\}$ we have $r_1 \cdot r_2 \cdot r_3 = \mathbf{0}_R$.

    Similar to the previous case, by Lemma \ref{length_2}, it is enough to check that for any paths $p_1, p_2$ in $G$ with same endpoints and satisfying $|p_1| = |p_2| = 2$ we have $l(p_1) = l(p_2)$.
    Since $t(e) \neq o(e')$, then for each path $p = (e_1, e_2)$ at least one of the inequalities $e_1 \neq e$, $e_2 \neq e'$ holds. Therefore, either $l(e_1) = \mathbf{0}_R$, or $l(e_2) = \mathbf{0}_R$, or $(e_1, e_2) = (e', e)$. In the first two cases we automatically have $l(p) = \mathbf{0}_R$. In the third case, we also have
    $$l(p) = l(e')l(e) = E_{2, 3} E_{1, 2} = \mathbf{0}_R.$$
    Therefore, $l(p_1) = \mathbf{0}_R = l(p_2)$. This concludes the proof.
\end{proof}

\begin{lemma} \label{efneqg}
    Let $G = (V, E)$ be a quasi-acyclic $2$-path-bounded graph without loops, multiple edges and triangles. Then for any three pairwise distinct edges $e, f, h \in E$ there exists a $3$-vanishing commutative diagram $D = (G, l)$ with values in a ring $R$ such that $l(e)l(f) \neq l(h)$.
\end{lemma}

\begin{proof}
    By Lemma \ref{NZedge} there exists a labeling $l$ such $l(h) \neq 0$ and $l(x) = 0$ for any edge $x \neq h$. Then $l(e)l(f) = \mathbf{0}_R \neq l(h)$.
\end{proof}

\begin{lemma} \label{claim_1}
    Let $G = (V, E)$ be a quasi-acyclic $2$-path-bounded graph without loops, multiple edges and triangles. Let $\RR$ be any minimal complete system of relations for $G$. Then for any $(s, s') \in \RR$ we have either $|s| = |s'| = 2$ or $|s| \geq 3$ and $|s'| \geq 3$.
\end{lemma}

\begin{proof}
    Assume the converse, then one of the following cases holds. Without loss of generality assume $|s| \geq |s'|$.

    \textit{Case 0.} $|s| = |s'| = 0$.

    This contradicts minimality of $\RR$, since this relation can be eliminated.
    
    \textit{Case 1.} $|s| \geq 1, \ |s'| = 0$.

    Consider labeling $l: E \to (\R, \cdot)$ such that $l(e) = 0$ for any $e \in E$. Then $\D = (G, l)$ is automatically commutative since $G$ does not contain any loops. At the same time, relation $s = s'$ is not satisfied since $l(s) = 0 \neq 1 = l(s')$. This is a contradiction.

    \textit{Case 2.} $|s| \geq 3, \ |s'| = 2$.

    Let $s' = (e, f)$. By Lemma \ref{NZef}, there exists a $3$-vanishing commutative diagram $\D = (G, l)$ with values in a ring $R$ such that $l(e)l(f) \neq \mathbf{0}_R = l(s)$. This is a contradiction. 

    \textit{Case 3.} $|s| \geq 3, \ |s'| = 1$.

    Let $s' = (e)$. By Lemma \ref{NZedge} there exists a $3$-vanishing commutative diagram $\D = (G, l)$ with values in a ring $R$ such that $l(e) \neq \mathbf{0}_R = l(s)$. This is a contradiction.

    \textit{Case 4.} $|s| = 2, \ |s'| = 1 $.

    Let $s = (e, f), s' = (g)$. By Lemma \ref{efneqg}, there exists a $3$-vanishing commutative diagram $\D = (G, l)$ with values in a ring $R$ such that $l(e)l(f) \neq l(g)$. This is a contradiction.

    \textit{Case 5.} $|s| = 1, \ |s'| = 1 $.

    Let $s = (e), s' = (e')$. Since $\RR$ is minimal, it follows that $e \neq e'$. By Lemma \ref{eneqf}, there exists a $3$-vanishing commutative diagram $\D = (G, l)$ with values in a ring $R$ such that $l(e) \neq l(e')$. Then we obtain $l(s) = l(e) \neq l(e') = l(s')$.
    This is a contradiction.
\end{proof}

Let us finish the proof of Proposition \ref{no_loops}.

\begin{proof}[Proof of Proposition \ref{no_loops}]

    We prove by a contradiction and assume that $\eta(G) < \Rh(G)$. Then there exists a minimal complete system of relations  $\RR$ for $G$ with  $|\RR| < \Rh(G)$.
    
    Since we have $\Rh(G)$ disjointed rhomboids in $G$, then there are $2 \cdot \Rh(G)$ pairwise disjoint consecutive pairs of edges appearing in these rhomboids. Since $|\RR| < \Rh(G)$, it follows that there exists a rhomboid $(a, b, c, d)$ such that for any relation $(s, s') \in \RR$ we have $s \neq (a, b)$ and $s' \neq (a, b)$. 
    
    We take $R = \mathrm{Mat}_2(\R)$ and consider the following labeling $l : E \to \mathrm{Mat}_3(\R)$

    $$l(x) = \begin{cases}
        E_{1, 2} & \text{if $x = a$}\\
        E_{2, 3} & \text{if $x = b$}\\
        \mathbf{0}_R & \text{otherwise.}
    \end{cases}
    $$
    Then $\D = (G, l)$ is not commutative since 
    $$l((a, b)) = E_{1, 2} E_{2, 3} = E_{1, 3} \neq \mathbf{0}_R = l((c, d)).$$
    
    Let us show that all relations in $\RR$ are satisfied. Note that $\D$ is 3-vanishing and for any pair $x, y \in E$ such that $(x, y) \neq (a, b)$ we have $l(x)l(y) = \mathbf{0}_R$. From Lemma \ref{claim_1} we know that for any $(s, s') \in \RR$ we have either $|s| \geq 3$ and $|s'| \geq 3$ or $|s| = |s'| = 2$. Since $\D$ is 3-vanishing, in the first case we have $l(s) = \mathbf{0}_R = l(s')$. In the second case, we have $s \neq (a, b)$ and $s' \neq (a, b)$, so the same equality $l(s) = \mathbf{0}_R = l(s')$ holds.
\end{proof}

Now let us finish the proof of Proposition \ref{no_loops_mult}.
\begin{proof}[Proof of Proposition \ref{no_loops_mult}]
    By Corollary \ref{minimal_nu}, there exists a minimal complete system of relations $\RR$ for $G$ such that $\nu(G) = m(\RR)$.
    Proposition \ref{no_loops} implies that $|\RR| \geq \Rh(G)$. By Lemma \ref{claim_1}, we know that for any $(s, s') \in \RR$ we have either $|s| = |s'| = 2$ or $|s| \geq 3$ and $|s'| \geq 3$. Then, by Corollary \ref{rb_is_r}, we obtain
    $$\nu(G) = m(\RR) \geq |\RR| \geq \Rh(G).$$
\end{proof}

\subsection{Reduction to loopless case} Let us extend our results from the previous subsection to all quasi-acyclic graphs. First we need several additional lemmas.

\begin{lemma} \label{lmln}
    Let $G$ be a quasi-acyclic $2$-path-bounded graph without multiple edges and triangles. Let $\RR$ be any complete system of relations for $G$. Denote by $\RR_L \subseteq \RR$ the subset consisting of all relations such that every edge in them is a loop. Denote by $\RR_N \subseteq \RR$ the subset consisting of all relations such that both their sides contain an edge which is not a loop. Then $\RR = \RR_L \sqcup \RR_N$.
\end{lemma}

\begin{proof} 
    Consider following labeling $l: E \to (\R, \cdot)$

    $$l(x) = \begin{cases}
        1 & \text{if $x$ is a loop}\\
        0 & \text{otherwise.}
    \end{cases}$$
    Let us check that $\D = (G, l)$ is commutative.
    Indeed, since $G$ is quasi-acyclic, then any path $p$ in $G$ with $o(p) = t(p)$ does not contain non-loop edges and so $l(p) = 1$. In the case $o(p) \neq t(p)$, the path $p$ contains a non-loop edge and hence $l(p) = 0$.
    
    Consider any relation $(s, s') \in \RR$. As $\D$ is commutative, we have $l(s) = l(s')$. If $l(s) = 0 = l(s')$, then both $s$ and $s'$ contain a non-loop edge, so $(s, s') \in \RR_N$. 
    Otherwise, $l(s) = 1 = l(s')$, so and both consist only of loops, i.e. $(s, s') \in \RR_L$.

    Therefore, $\RR$ can be represented as $\RR = \RR_L \sqcup \RR_N$.
\end{proof}

\begin{lemma} \label{rl_lg}
    Let $G$ be a quasi-acyclic $2$-path-bounded graph without multiple edges and triangles. Let $\RR$ be any complete system of relations for $G$. Denote by $\RR_L \subseteq \RR$ the subset consisting of all relations such that every edge in them is a loop. Then $|\RR_L| \geq \L(G)$.
\end{lemma}

\begin{proof}
    Let us prove the Lemma by contradiction. Assume that $|\RR_L| < \L(G)$. Take $M = \Mat_2(\R)$ with respect to the matrix multiplication. It suffices to construct a labeling $l: E \to \Mat_2(\R)$ such that $\D = (G, l)$ is not commutative and all the relations from $\RR$ are satisfied.

    Let $e_1, \dots e_{|\L(G)|} \in E$ be the loops of $G$. Let $\RR_L = \{R_1, \dots, R_{|\RR_L|}\}$, where $R_i = (s_i, s'_i)$.
    Consider the matrix $A = (a_{i, j}) \in \Mat_{|\RR_L| \times \L(G)}$ defined as follows. We set 
    $$a_{i, j} = (\text{number of occurrences of $e_j$ in $s_i$}) - (\text{number of occurrences of $e_j$ in $s'_i$}).$$
    Since $|\RR_L| < \L(G)$, there exists $0 \neq v \in \ker A \subseteq \R^{|\L(G)|}$. Let $v = \{v_1, \dots, v_{|\L(G)|}\}^T$.
    Consider the map $l: E \to \Mat_2(\R)$ given by
    $$l(x) = \begin{cases}
    \begin{pmatrix} 
            1 & v_j \\
            0 & 1
    \end{pmatrix} & \text{if $x = e_j$, where $1 \leq j \leq |\L(G)|$}, \\
    \;\;\;\;\;\; \mathbf{0}_M & \text{otherwise}.
    \end{cases}.$$
    
    Since $v \neq 0$ then there exists $j$ such that $v_j \neq 0$. Then $l(e_j) \neq \mathbf{1}_M$ and $\D$ is not commutative because $e_j$ is a loop.
    
    It remains to show that all the relations from $\RR$ are satisfied. By Lemma \ref{lmln}, we have $\RR = \RR_L \sqcup \RR_N$. Every relation from $\RR_N$ is automatically satisfied since every its side contains a non-loop edge and hence its value equals to $\mathbf{0}_M$.
    
    Finally, let us take any relation $(s_i, s'_i) \in \RR_L$. Since for any $x, y \in \R$ we have
    $$  \begin{pmatrix} 
            1 & x \\
            0 & 1
    \end{pmatrix} \cdot  \begin{pmatrix} 
            1 & y \\
            0 & 1
    \end{pmatrix} =  \begin{pmatrix} 
            1 & x + y \\
            0 & 1
    \end{pmatrix},$$
    then we have 
    $$l(s_i) = \begin{pmatrix} 
            1 & v_{s_i} \\
            0 & 1
    \end{pmatrix},$$
    where 
    $$v_{s_i} = \sum_{j=1}^{|\L(G)|} v_j \cdot (\text{number of occurrences of $e_j$ in $s_i$}).$$
    Similarly, $$l(s'_i) = \begin{pmatrix} 
            1 & v_{s'_i} \\
            0 & 1
    \end{pmatrix},$$
    where 
    $$v_{s'_i} = \sum_{j=1}^{|\L(G)|} v_j \cdot (\text{number of occurrences of $e_j$ in $s'_i$}).$$
    We obtain
    $$v_{s_i} - v_{s'_i} = \sum_{j=1}^{|\L(G)|} v_j \cdot (\text{number of occurrences of $e_j$ in $s_i$}) - $$
    $$ - \sum_{j=1}^{|\L(G)|} v_j \cdot (\text{number of occurrences of $e_j$ in $s'_i$}) = $$
    $$ = \sum_{j=1}^{|\L(G)|} v_j \cdot \Bigl( (\text{number of occurrences of $e_j$ in $s_i$}) - (\text{number of occurrences of $e_j$ in $s'_i$}) \Bigr) =$$
    $$ = \sum_{j=1}^{|\L(G)|} v_j \cdot a_{i, j} =  (Av)_i = 0, $$
    where the last equality follows from the fact that $v \in \ker A$. Therefore, $v_{s_i} = v_{s'_i}$ and hence $l(s) = l(s')$. This concludes the proof.
\end{proof}

\begin{lemma} \label{const_reduction}
    Let $G$ be a quasi-acyclic $2$-path-bounded graph without multiple edges and triangles. Let $\RR$ be any complete system of relations for $G$. Denote by $G' = (V, E')$ the graph obtained from $G$ by removing all loop-edges and let $\RR'_N$ be a family of relations obtained from $\RR_N$ by removing all loop-edges. Then $\RR'_N$ is a complete system of relations for $G'$.
\end{lemma}

\begin{proof}
    First, let us show that for any commutative diagram $\D' = (G', l')$ all the relations from $\RR'_N$ are satisfied. Indeed, we can extend $\D'$ to a commutative diagram $\D = (G, l)$ by setting $l(e) = \mathbf{1}_M$ for any $e \in E \setminus E'$. Note that the relations $\RR_N$ are satisfied for $\D$. Moreover, by construction, the values of the sides of the relations $\RR_N$ on $\D$ coincide with the correspondent values of $\RR'_N$ on $D'$. Hence all the relations from $\RR'_N$ are satisfied for $\D'$.

    Now let us prove in the same way that $\RR'_N$ is complete. Assume the converse and let $\D' = (G', l')$ be a non-commutative diagram such that the relations $\RR'_N$ are satisfied. Again we can extend $\D'$ to a non-commutative diagram $\D = (G, l)$ by setting $l(e) = \mathbf{1}_M$ for any $e \in E \setminus E'$. Similarly, the values of the sides of the relations $\RR_N$ on $\D$ coincide with the correspondent values of $\RR'_N$ on $\D'$. Hence all the relations from $\RR_N$ are satisfied for $\D$. Since $l(e) = \mathbf{1}_M$ for any $e \in E \setminus E'$, all the relations from $\RR_L$ are also satisfied for $\D$. Therefore, all the relations $\RR$ are satisfied for a non-commutative diagram $\D$. This contradiction proves that $\RR'_N$ is complete for $G'$.
\end{proof}

\begin{corollary} \label{nu_reduction}
    Let $G$ be a quasi-acyclic $2$-path-bounded graph without multiple edges and triangles. Denote by $G' = (V, E')$ the graph obtained from $G$ by removing all loop-edges. Then $\nu(G') \leq \nu(G)$.
\end{corollary}

\begin{proof}
    By Corollary \ref{minimal_nu}, there exists a complete system relations $\RR$ for $G$ such that $\nu(G) = m(\RR)$. Let $\RR'_N$ be a family of relations obtained from $\RR_N$ by removing all loop-edges. By Lemma \ref{const_reduction} we know that $\RR'_N$ is a complete system of relations for $G'$.
    Since by Lemma \ref{m_momotonous} and Corollary \ref{nu_is_monotonous} we have 
    $$m(\RR'_N) \leq m(\RR_N) \leq m(\RR),$$ 
    then we obtain
    $$\nu(G') \leq m(\RR'_N) \leq m(\RR) = \nu(G),$$
    which concludes the proof.
\end{proof}

\begin{lemma} \label{rn_rg}
    Let $G$ be a quasi-acyclic $2$-path-bounded graph without multiple edges and triangles. Let $\RR$ be any complete system of relations for $G$. Denote by $\RR_N \subseteq \RR$ the subset consisting of all relations such that both their sides contain an edge which is not a loop. Then $|\RR_N| \geq \Rh(G)$.
\end{lemma}

\begin{proof}
    Denote by $G' = (V, E')$ the graph obtained from $G$ by removing all loop-edges and let $\RR'_N$ be a family of relations obtained from $\RR_N$ by removing all loop-edges. By Lemma \ref{const_reduction} we know that $\RR'_N$ is a complete system of relations for $G'$.

    Since $\RR'_N$ is a complete system of relations for $G'$, Proposition \ref{no_loops} implies that 
    $$|\RR_N| = |\RR'_N| \geq \Rh(G') = \Rh(G).$$ 
    The last equality follows from the fact that loops does not affect on the maximal number of disjoint rhomboids in a graph.
\end{proof}

Now we are able to finish the proof of the main results of this section.

\begin{proof}[Proof of Theorem \ref{eq_amount_theorem}]
Let $\RR$ be any complete system of relations for $G$.
Denote by $\RR_L \subseteq \RR$ the subset consisting of all relations such that every edge in them is a loop. Denote by $\RR_N \subseteq \RR$ the subset consisting of all relations such that both their sides contain an edge which is not a loop. 

Lemmas \ref{lmln}, \ref{rl_lg}, \ref{rn_rg} immediately imply 
$$|\RR| = |\RR_N| + |\RR_L| \geq \Rh(G) + \L(G),$$
which concludes the proof of Theorem \ref{eq_amount_theorem}.
\end{proof}

\begin{proof}[Proof of Theorem \ref{mult_amount_theorem}]
     Denote by $G'$ the graph obtained from $G$ by removing all loop-edges.
     Combining Proposition \ref{no_loops_mult} with Corollary \ref{nu_reduction} we immediately obtain
    $$ \nu(G) \geq \nu(G') \geq \Rh(G') = \Rh(G).$$
\end{proof}

\section{Commutativity verification algorithm} \label{S6}

In this section we propose algorithm with no more than
$$\min(|V|^2, |E|) \cdot \min(|V|, |E|) + |E|$$
checks for equality and no more than 
$$\min(|V|^2, |E|) \cdot \min(|V|, |E|)$$
multiplications of elements of $M$. The algorithm begins with two preliminary steps: removing loops and multiple edges.

\subsection{Removing loops.} For each edge $e$ we check if $o(e) = t(e)$. If condition holds, the edge is identified as a loop, and we verify whether it equals $1_M$. Then we remove the edge $e$ from $G$.

    \begin{algorithm}[H]
    \caption{Removing Loops}
    \begin{algorithmic}[1] 
    \REQUIRE Graph $G$, labeling $l$.
    \STATE \textbf{procedure} RemoveLoops($G$)
    \FOR{each vertex $v$ in $G$}
        \FOR{each edge $e$ outgoing from $v$}
            \IF{$t(e) = v$}
                \IF{$l(e) \neq \mathbf{1}_M$}
                    \STATE \textbf{exit} with value \texttt{False}
                \ENDIF
                
                \STATE \textbf{remove} $e$ from $G$
            \ENDIF
        \ENDFOR
    \ENDFOR
    \end{algorithmic}
    \end{algorithm}

\subsection{Removing multiple edges.} For each vertex $v$, we examine all of its outgoing edges. If two edges are found to have the same tail, we check if they are identical, and then remove the second edge of the pair from $G$.

    \begin{algorithm}[H]
    \caption{Removing Multiple Edges}
    \begin{algorithmic}[1] 
    \REQUIRE Graph $G$, labeling $l$.
    \STATE \textbf{procedure} RemoveMultipleEdges($G$)
    \FOR{each vertex $v$ in $G$}
        \STATE \textbf{create array} $T$ of length $|V|$ with empty values
        \FOR{each edge $e$ outgoing from $v$}
            \IF{$T[t(e)]$ is empty}
                \STATE \textbf{set} $T[t(e)] = l(e)$
            \ELSE
                \IF{$l(e) \neq T[t(e)]$}
                    \STATE \textbf{exit} with value \texttt{False}
                \ENDIF

                \STATE \textbf{remove} $e$ from $G$
            \ENDIF
        \ENDFOR
    \ENDFOR
    \end{algorithmic}
    \end{algorithm}

Note that these two steps combined give us no more then $|E|$ checks for equality, as multiple edge can not be a loop.
After these steps, we obtain a reduced graph $G = (V, E')$ satisfying $|E'| < |V|^2$, as there are no multiple edges or loops remaining. 
With the graph now simplified, we can proceed to the core part of the algorithm.

\subsection{DFS procedure.} Let us briefly recall the standard DFS algorithm. It will be conviniet to use following modification, where we define the functions \texttt{ProceedNewVertex} and \texttt{ProceedOldVertex} later.

    \begin{algorithm}[H]
    \caption{Depth-First Search}
    \begin{algorithmic}[1] 
    \REQUIRE Graph $G$, starting vertex $v$, labeling $l$.
    \STATE \textbf{procedure} DFS($G$, $v$)
    \STATE \textbf{mark} $v$ as visited
    \FOR{each edge $e$ outgoing from $v$}
        \STATE \textbf{set} $u = t(e)$ 
        \IF{$u$ is not visited}
            \STATE \texttt{ProceedNewVertex}$(v, e, u)$
            \STATE DFS($G$, $u$)
        \ELSE 
            \STATE \texttt{ProceedOldVertex}$(v, e, u)$
        \ENDIF
    \ENDFOR
    \end{algorithmic}
    \end{algorithm}

In order to define the functions \texttt{ProceedNewVertex} and \texttt{ProceedOldVertex}, for every vertex $v$ we keep an element of $M$ called $m(v)$. Before the algorithm starts, we set $m(v) = \mathbf{1}_M$ for all $v \in V$. Now we can define these functions.

    \begin{algorithm}[H]
    \caption{ProceedNewVertex}
    \begin{algorithmic}[1] 
    \REQUIRE An edge $e$, vertices $v = o(e)$, $u = t(e)$, labeling $l$.
    \STATE \textbf{procedure} ProceedNewVertex$(v, e, u)$
        \STATE \textbf{set} $m(u) = m(v) l(e)$ 
    \end{algorithmic}
    \end{algorithm}

    \begin{algorithm}[H]
    \caption{ProceedOldVertex}
    \begin{algorithmic}[1] 
    \REQUIRE An edge $e$, vertices $v = o(e)$, $u = t(e)$, labeling $l$.
    \STATE \textbf{procedure} ProceedOldVertex$(v, e, u)$
        \IF{$m(u) \neq m(v) l(e)$}
            \STATE \textbf{exit} with value \texttt{False}
        \ENDIF
    \end{algorithmic}
    \end{algorithm}

Finally, we define the core algorithm we run for a general graph $G$.

    \begin{algorithm}[H]
    \caption{Core Algorithm}
    \begin{algorithmic}[1] 
    \REQUIRE Graph $G$, labeling $l$.
    \STATE \textbf{procedure} CoreAlgorithm($G$)
        \STATE RemoveLoops($G$)
        \STATE RemoveMultipleEdges($G$)
        \FOR{each vertex $v$ in $V$}
            \STATE \textbf{DFS}($G$, $v$)
        \ENDFOR
        \STATE \textbf{exit} with value \texttt{True}
    \end{algorithmic}
    \end{algorithm}

\subsection{Proof} Let us prove that our algorithm indeed verifies diagram commutativity.
Let $\D = (G, l)$ be a diagram. By construction, the core algorithm returns \texttt{True} if $D$ is commutative.
It suffices to prove that for any vertex $v$ the algorithm \texttt{DFS}($G, v$) verifies commutativity of all pairs of paths starting in $v$. 

For every vertex $u \in V$ we denote by $p_u$ the unique simple path from $v$ to $u$ on the DFS tree of the algorithm \texttt{DFS}($G, v$). By definition of the function \texttt{ProceedNewVertex} we have $m(u) = l(p_u)$ for any $u \in V$. So it suffices to check that for any path $p$ with $o(p) = v$ and $t(p) = u$ we have $l(p) = l(p_u)$.

We prove this by induction on the length $|p|$ of path $p$. Base of induction for $|p| = 0$ is trivial. Let us prove the induction step. Let $e$ be the last edge of $p$, and denote by $p'$ the path obtained from $p$ by removing $e$. We have $l(p) = l(p')l(e)$.
Let $u' = t(p')$.
Since $|p'| = |p| - 1$, then by the induction hypothesis we have $l(p') = l(p_{u'}) = m(u')$.

Since $u'$ is reachable from $v$, then the function \texttt{ProceedOldVertex}$(u', e, u)$ was called and did not returned \texttt{False}. Hence we have $m(u) = m(u')l(e)$. Therefore,
$$l(p) = l(p')l(e) = m(u')l(e) = m(u).$$
This concludes the proof.

\section{Time complexity} \label{S7}

In this section, we prove Theorems \ref{upper_bound} and \ref{upper_bound_mult}. In particular, we calculate the number of multiplications and equality checks in our algorithm. We also compute the asymptotics of the whole algorithm itself.

Denote by $E'$ the set of edges of $G$ after removing multiple edges and loops. We start with the following lemma.

\begin{lemma} \label{amount_calls}
     The total amount of calls of functions \texttt{ProceedNewVertex} and \\ \texttt{ProceedOldVertex} is no more than 
     $$|E'| \cdot \min(|V|, |E'| + 1).$$
\end{lemma}

\begin{proof}
    The total number of calls is equal to the number of pairs $(v, e) \in V \times E'$, such that \texttt{DFS}$(G, v)$ reaches $e$. This number is less or equal to 
    $$\sum_{e \in E'}|C_e| \leq \sum_{e \in E'} \min(|V|, |E'| + 1) = |E'| \cdot \min(|V|, |E'| + 1),$$
    where $C_e \subseteq V$ is the connected component of $e$. 
\end{proof}

\begin{theorem} \label{algo_mult}
    The core algorithm performs no more than 
    $$|E'| \cdot \min(|V|, |E'| + 1) \leq \min(|V|^2, |E|) \cdot \min(|V|, |E| + 1)$$
    multiplication operations in $M$.
\end{theorem}

\begin{proof}
    Note that the multiplication operation only arises once in both \texttt{ProceedOldVertex} and \texttt{ProceedNewVertex} functions, hence the amount of multiplications is equal to amount of their calls. Therefore Lemma \ref{amount_calls} implies the required inequality.
\end{proof}

\begin{theorem} \label{algo_eq}
    The above algorithm performs no more than 
    $$|E'| \cdot \min(|V|, |E'| + 1) + |E| \leq \min(|V|^2, |E|) \cdot \min(|V|, |E| + 1) + |E|$$
    equality checks in $M$.
\end{theorem}

\begin{proof}
    First, let us show that number of equality checks inside the functions $\texttt{RemoveLoops}$ and $\texttt{RemoveMultipleEdges}$ calls is no more than $|E|$. Indeed, if $e \in E$ is a loop then it corresponds to exactly one equality check in the line $5$ of $\texttt{RemoveLoops}$ function and then it will be removed. If $e \in E$ is multiple (but not a loop) then it corresponds to no more than one equality check in line $8$ of $\texttt{RemoveMultipleEdges}$ function. Other edges do not arise any equality checks inside these calls. 

    In \texttt{DFS} algorithms, the equality check only arises once in the \texttt{ProceedOldVertex} function, hence the amount of equality checks is equal to amount of its calls, which is no more than $|E'| \cdot \min(|V|, |E'| + 1)$ by Lemma \ref{amount_calls}. 

    Since $|E'| \leq |V|^2$, the total number of equality checks is no more than 
    
    $$|E'| \cdot \min(|V|, |E'| + 1) + |E| \leq \min(|V|^2, |E|) \cdot \min(|V|, |E| + 1) + |E|.$$
\end{proof}

\begin{corollary} \label{main_cor}
    The above algorithm is equivalent to verification of some complete system of relations $\RR$ for $G$ with 
    $$|\RR| \leq \min(|V|^2, |E|) \cdot \min(|V|, |E| + 1) + |E|$$
    and 
    $$m(\RR) \leq \min(|V|^2, |E|) \cdot \min(|V|, |E| + 1).$$
    In particular, such a complete system of relations $\RR$ for $G$ exists.
\end{corollary}

\begin{proof}
    Let us construct a complete system of relations $\RR$ with required properties for a graph $G$ using our algorithm. We start with the set $S = \mathcal{E}$ containing a one-edge sequence $(e)$ for each $e \in E$. At every step the algorithm can verify the equality $l(s) = l(s')$ for some $s, s' \in S$ or multiply $l(s)$ by $l(s')$ for some $s, s' \in S$. In first case, we add the relation $(s, s')$ to $\RR$. In the second case we add $s \circ s'$ to $S$.

    Therefore, Theorems \ref{algo_mult}, \ref{algo_eq} imply the required inequalities.
\end{proof}

\begin{theorem} \label{full_comlexity}
    The total time complexity of the above algorithm is 
    \vspace{-10pt}
    
    $$O \left( \min(|V|^2, |E|) \cdot \min(|V|, |E|) + |E| \right) \cdot T_{\mathrm{equal}}  \ +$$ 
    
    \vspace{-10pt}
    $$\hspace{8pt} O \left( \min(|V|^2, |E|) \cdot \min(|V|, |E|) \right) \cdot T_{\mathrm{multi}} + O(|V|),$$ 
    
\end{theorem}

\begin{proof}
    In view of Corollary \ref{main_cor}, it remains to estimate the complexity of the algorithm without multiplications and equality checks. Let us do it step by step.
    
    \smallskip

    \textbf{\texttt{RemoveLoops} function.} The complexity is $O(|V| + |E|)$.

    \smallskip

    \textbf{\texttt{RemoveMultipleEdges} function.} The complexity is $O(|V| + |E|)$.

    \smallskip

    \textbf{All \texttt{DFS} calls combined.} The complexity is $O(|V|) + O(q)$, where $q$ is a number of calls of $\texttt{RemoveLoops}$ and $\texttt{RemoveMultipleEdges}$ functions, which is no more than 
    $$|E'| \cdot \min(|V|, |E'| + 1) \leq \min(|V|^2, |E|) \cdot \min(|V|, |E| + 1)$$ 
    by Lemma \ref{amount_calls}.

    \smallskip
    
    The total remaining complexity is
    $$O(|V| + |E|) + O(\min(|V|^2, |E|) \cdot \min(|V|, |E| + 1)) =$$
    $$= O(|V|) + O(\min(|V|^2, |E|) \cdot \min(|V|, |E|)).$$
    This concludes the proof.
\end{proof}

Now, the proof of our main upper-bound results is straightforward.

\begin{proof}[Proof of Theorem \ref{upper_bound}]
    Immediately follows from Corollary \ref{main_cor}.
\end{proof}

\begin{proof}[Proof of Theorem \ref{upper_bound_mult}]
    Immediately follows from Corollary \ref{main_cor}.
\end{proof}

\appendix

\section{Proof of Theorem \ref{nu_ge}} \label{appendix}

In order to prove Theorem \ref{nu_ge}, for any $n, m \geq 4$, we explicitly construct a graph $G = (V, E)$ with $|V| = n$ and $|E| = m$, satisfying 
\begin{equation} \label{triploid_relation}
    \Rh(G) + \L(G) \geq 2^{-14} \cdot \left(\min(m, n^2) \cdot \min(m, n) + m\right),
\end{equation}
and
\begin{equation} \label{triploid_second_relation}
\Rh(G) \geq 2^{-14} \cdot \left(\min(m, n^2) \cdot \min(m, n)\right),
\end{equation}
such that $G \cong T(n_1, n_2, n_3, n_0, e)$ for some integers $n_1, n_2, n_3, n_0, e$.

We brake the proof of Theorem \ref{nu_ge} into several cases. 

$\bullet$ $m \leq 16$, or $n \leq 16$ and $m \leq n^2$,

$\bullet$ $16 < m \leq 2 \cdot n - 4$,

$\bullet$ $n, m > 16$ and $2 \cdot n - 4 < m \leq n^2$,

$\bullet$ $n^2 < m$.

There cases are covered by following Lemmas \ref{m16}, \ref{mlown}, \ref{mn}, and \ref{mbign}, respectively. For each of this cases we prove inequalities (\ref{triploid_relation}) and (\ref{triploid_second_relation}) for a special explicitly constructed triploid.

\begin{lemma} \label{m16}
    For any integers $n, m \geq 4$ satisfying either $m \leq 16$, or $n \leq 16$ and $m \leq n^2$, there exists an oriented graph $G = (V, E)$ with $|V| = n$ and $|E| = m$ such that inequalities (\ref{triploid_relation}) and (\ref{triploid_second_relation}) hold and such that $G \cong T(n_1, n_2, n_3, n_0, e)$ for some integers $n_1, n_2, n_3, n_0, e$. 
\end{lemma}

\begin{proof}

    Let us take \(G = T(1, 2, 1, n - 4, m)\). It is well-defined since $m \ge 4 = 2 \cdot (1 + 1)$. Note that $G$ contains exactly one rhomboid. 
    Hence we obtain
    \begin{equation} \label{eq:m16}
        \Rh(G) = 1 \geq 2^{-14} \cdot \left(256 \cdot 16 + 256\right) \geq 2^{-14} \cdot \left(\min(m, n^2) \cdot \min(m, n) + m\right).
    \end{equation}
    The last inequality in (\ref{eq:m16}) holds, since we have $m \le 256$. Both inequalities (\ref{triploid_relation}) and (\ref{triploid_second_relation}) automatically follow from (\ref{eq:m16}).
\end{proof}

\begin{lemma} \label{mlown}
    For any integers $n, m \geq 4$ satisfying $16 < m \leq 2 \cdot n - 4$ there exists an oriented graph $G = (V, E)$ with $|V| = n$ and $|E| = m$ such that inequalities (\ref{triploid_relation}) and (\ref{triploid_second_relation}) hold and such that $G \cong T(n_1, n_2, n_3, n_0, e)$ for some integers $n_1, n_2, n_3, n_0, e$. 
\end{lemma}

    \begin{proof}
    Let us take \(G = T\left(\left\lfloor \frac{m}{4} \right\rfloor, 2, \left\lfloor \frac{m}{4} \right\rfloor, n - 2 \cdot \left\lfloor \frac{m}{4} \right\rfloor - 2, m\right)\). 
    It is well-defined since
    $$m \ge 2 \cdot \left( \left\lfloor \frac{m}{4} \right\rfloor + \left\lfloor \frac{m}{4} \right\rfloor \right).$$
    Also note that $n - 2 \cdot \left\lfloor \frac{m}{4} \right\rfloor - 2 \geq 0$ since $m \leq 2 \cdot n - 4$.
    
    By Lemma \ref{triploid_rhombs}, $G$ contains at least \(\left\lfloor \frac{m}{4} \right\rfloor^2\) disjointed rhomboids. Hence we obtain
    \begin{subequations}\label{eq:mlown}
      \makeatletter
      \renewcommand{\theequation}{\arabic{parentequation}.\arabic{equation}}
      \makeatother
    
      \begin{align}
        \Rh(G) & \geq \left\lfloor \frac{m}{4} \right\rfloor^2 \notag \\
        & \geq \frac{(m - 3)^2}{16} \notag \\
        & \geq \frac{m^2}{64} \label{eq:mlown:1} \\ 
        &\geq \frac{m^2 + m}{128} \label{eq:mlown:2} \\
        &\geq \frac{\min(m, n^2) \cdot \min(m, n) + m}{128} \notag \\
        &\geq 2^{-14} \cdot \left(\min(m, n^2) \cdot \min(m, n) + m\right). \notag
      \end{align}
    \end{subequations}
    Here (\ref{eq:mlown:1}) and (\ref{eq:mlown:2}) follows from the assumption $m > 16$.
    Both inequalities (\ref{triploid_relation}) and (\ref{triploid_second_relation}) automatically follow from (\ref{eq:mlown}).
    \end{proof}

\begin{lemma} \label{mn}
    For any integers $n, m > 16$ satisfying $2 \cdot n - 4 < m \leq n^2$ there exists an oriented graph $G = (V, E)$ with $|V| = n$ and $|E| = m$ such that inequalities (\ref{triploid_relation}) and (\ref{triploid_second_relation}) hold and such that $G \cong T(n_1, n_2, n_3, n_0, e)$ for some integers $n_1, n_2, n_3, n_0, e$. 
\end{lemma}

\begin{proof}
    Let us take \(G = T\left(\left\lfloor \frac{n - t}{4} \right\rfloor, t, \left\lfloor \frac{n - t}{4} \right\rfloor, n - 2 \cdot \left\lfloor \frac{n - t}{4} \right\rfloor - t, m \right)\) where
    \[
    s = \sqrt{\max(0, n^2 - 4m)}
    \]
    and
    \[
    t = \left\lceil \frac{n - s}{2} \right\rceil.
    \]
    Note that $n - 2 \cdot \left\lfloor \frac{n - t}{4} \right\rfloor - t \geq 0$ since

    \begin{subequations}\label{eq:chain}
      \makeatletter
      \renewcommand{\theequation}{\arabic{parentequation}.\arabic{equation}}
      \makeatother

      \begin{align}
          n - 2 \cdot \left\lfloor \frac{n - t}{4} \right\rfloor - t & \geq n - \left\lfloor \frac{n - t}{2} \right\rfloor - t \notag \\ & \geq n - (n - t) - t \notag \\ 
          & = 0. \notag
      \end{align}
    \end{subequations}
    We also need that ${s^2 < (n - 2)^2}$. \newline
    \textbf{Case 1.} $n^2 - 4m \leq 0$.
    $$s^2 = \max(0, n^2 - 4m) = 0 < (n - 2)^2.$$
    \textbf{Case 2.} $n^2 - 4m > 0$.
    
    \begin{subequations}\label{eqs2:chain}
      \makeatletter
      \renewcommand{\theequation}{\arabic{parentequation}.\arabic{equation}}
      \makeatother

      \begin{align}
            s^2 = \max(0, n^2 - 4m) & 
            = n^2 - 4m \notag \\ 
            &<  n^2 - 4(n - 1) \label{eqs2:chain:1} \\
            & = (n - 2)^2.
      \end{align}
    \end{subequations}

    Note that (\ref{eqs2:chain:1}) holds since by assumption $m > 2 \cdot n - 4 > n - 1$ because $n \geq 4$. Therefore, $s < n - 2$ and

    \begin{equation}
        \label{ns_geq_2}
        \dfrac{n - s}{2} > 1.
    \end{equation}
    
    Let us show that $G$ is well-defined. Indeed, we have
    \begin{subequations}\label{eq:chain}
      \makeatletter
      \renewcommand{\theequation}{\arabic{parentequation}.\arabic{equation}}
      \makeatother

      \begin{align}
          t \cdot \left( \left\lfloor \frac{n - t}{4} \right\rfloor + \left\lfloor \frac{n - t}{4} \right\rfloor \right) & \leq t \cdot \frac{(n - t)}{2} = \frac{n \cdot t - t^2}{2} \notag \\
          & = \dfrac{1}{2} \cdot \left( n \cdot \left\lceil \frac{n - s}{2} \right\rceil - \left\lceil \frac{n - s}{2} \right\rceil^2 \right) \notag \\
          & \leq \dfrac{1}{2} \cdot \left( n \cdot \left\lceil \frac{n - s}{2} \right\rceil - \left( \frac{n - s}{2} \right) ^2 \right) \notag \\
          & = \dfrac{1}{2} \cdot  n \cdot \left\lceil \frac{n - s}{2} \right\rceil - \dfrac{1}{2} \cdot \left( \frac{n - s}{2} \right) ^2  \notag \\
          & \leq n \cdot \dfrac{1}{2} \cdot \left\lceil \frac{n - s}{2} \right\rceil - \left( \frac{n - s}{2} \right) ^2  \notag \\
          & \leq n \cdot \frac{n - s}{2}  - \left( \frac{n - s}{2} \right) ^2 \label{eqnt:1} \\
          & =  \frac{n^2 - s\cdot n}{2} -  \frac{n^2 + s^2 - 2\cdot n \cdot s}{4} \notag \\
          & = \frac{n^2 - 2\cdot s \cdot n - (n^2 - 4m) + 2\cdot n \cdot s}{4} \notag \\ 
          & = \frac{2\cdot n \cdot s - 2 \cdot s \cdot n + 4m}{4} = \frac{4m}{4}\notag \\
          & = m \notag .
      \end{align}
    \end{subequations} 

    Note that inequality (\ref{eqnt:1}) holds since

    \begin{subequations}\label{eqs2:chain}
      \makeatletter
      \renewcommand{\theequation}{\arabic{parentequation}.\arabic{equation}}
      \makeatother

      \begin{align}
            \dfrac{1}{2} \cdot \left\lceil \frac{n - s}{2} \right\rceil & 
            \leq \dfrac{1}{2} \cdot \left( \dfrac{n - s}{2} + 1 \right) \notag \\
            & = \dfrac{1}{2} \cdot \dfrac{n - s}{2} + \dfrac{1}{2} \notag \\
            & \leq \dfrac{1}{2} \cdot \dfrac{n - s}{2} + \dfrac{1}{2} \cdot \dfrac{n - s}{2} \label{eqns2}
            \\
            & = \dfrac{n - s}{2}. \notag
      \end{align}
    \end{subequations}

    The relation (\ref{eqns2}) holds because $\dfrac{n - s}{2} > 1$ by inequality (\ref{ns_geq_2}).

    By Lemma \ref{triploid_rhombs}, we have at least $\left\lfloor \frac{n - t}{4} \right\rfloor^2 \cdot \left\lfloor \frac{t}{2} \right\rfloor$ rhomboids. 
    
    \textbf{Case 1.} \(n^2 - 4m \geq 1\).  
    
     \begin{subequations}\label{eq:chain}
      \makeatletter
      \renewcommand{\theequation}{\arabic{parentequation}.\arabic{equation}}
      \makeatother
    
      \begin{align}
        \Rh(G) \geq \left\lfloor \frac{n - t}{4} \right\rfloor^2 \cdot \left\lfloor \frac{t}{2} \right\rfloor
        &\ge \left(\frac{n - t}{8}\right)^2 \cdot \frac{t}{4}
          = \frac{(n - t)^2 \cdot t}{256}
           \label{eq:chain:22}\\
        & =  \frac{(n - \left\lceil \frac{n - s}{2} \right\rceil)^2 \cdot \left\lceil \frac{n - s}{2} \right\rceil}{256}
          \notag \\
        & \geq  \frac{(n - \left\lceil \frac{n - s}{2} \right\rceil)^2 \cdot \frac{n - s }{2}}{256}
          \notag\\
        &\geq \frac{\bigl(n - \tfrac{n - s + 2}{2}\bigr)^2 \cdot \tfrac{n - s}{2}}{256}
          \notag\\
        &= \frac{\bigl(\tfrac{n + s - 2}{2}\bigr)^2 \cdot \tfrac{n - s}{2}}{256}
          \notag\\
        &= \frac{(n + s - 2)^2 (n - s)}{2048}
          \notag\\
        &\ge \frac{(n + s)^2 (n - s)}{2^{13}}
          = \frac{(n + s)(n^2 - s^2)}{2^{13}}
          \notag \\
        &= \frac{(n + \sqrt{\max(0, n^2 - 4m)})(n^2 - \max(0, n^2 - 4m))}{2^{13}}
      \notag\\
      &= \frac{(n + \sqrt{n^2 - 4m})(n^2 - n^2 + 4m)}{2^{13}}
      \notag\\
      &= \frac{(n + \sqrt{n^2 - 4m}) \cdot 4m}{2^{13}}
      \notag\\
    &= \frac{n\,m + m\,\sqrt{n^2 - 4m}}{2^{11}}
      \notag\\
    &\ge \frac{n\,m + m}{2^{11}}
      \label{eq:chain:1}\\
    &\ge 2^{-14} \cdot \left(\min(m, n^2) \cdot \min(m, n) + m\right).
       \notag
    \end{align}
    \end{subequations}

    Inequality (\ref{eq:chain:22}) holds since 
         \begin{subequations}\label{eq2:chain}
      \makeatletter
      \renewcommand{\theequation}{\arabic{parentequation}.\arabic{equation}}
      \makeatother

        \begin{align}
            n - t & = n - \left\lceil \frac{n - s}{2} \right\rceil \notag \\ & \geq n - \left\lceil \frac{n}{2} \right\rceil \notag \\ & = \left\lfloor \frac{n}{2} \right\rfloor \notag \\ & \geq 8. \notag
        \end{align}  
    \end{subequations}

     Inequality (\ref{eq:chain:1}) holds, since by assumption of this case we have \(n^2 - 4m \geq 1\). 
    
    \textbf{Case 2.} \(n^2 - 4m < 1\). 
    
    Since $n^2 - 4m$ is an integer, that means that $n^2 - 4m \leq 0$, therefore \(t = \left\lceil \frac{n}{2} \right\rceil \):

     \begin{subequations}\label{eq2:chain}
      \makeatletter
      \renewcommand{\theequation}{\arabic{parentequation}.\arabic{equation}}
      \makeatother

        \begin{align}
            \Rh(G) \geq \left\lfloor \frac{n - t}{4} \right\rfloor^2 \cdot \left\lfloor \frac{t}{2} \right\rfloor
        &\ge \left(\frac{n - t}{8}\right)^2 \cdot \frac{t}{4}
          = \frac{(n - t)^2 \cdot t}{256}
          \notag\\
            & = \frac{(n - \left\lceil \frac{n}{2} \right\rceil)^2 \cdot \left\lceil \frac{n}{2} \right\rceil}{256} \notag \geq \frac{(n - \left\lceil \frac{n}{2} \right\rceil)^2 \cdot  \frac{n}{2}}{256} \notag \\
            & \geq \frac{(n - \frac{n + 1}{2} )^2 \cdot  \frac{n}{2}}{256} \notag = \frac{(\frac{n - 1}{2} )^2 \cdot  \frac{n}{2}}{256} \notag \\
            & = \frac{(n - 1)^2 \cdot n}{2048} \notag \\
            &\ge \frac{n^3}{2^{13}} \notag \\
            &\ge \frac{n \cdot m}{2^{13}} \label{eq2:chain:3} \\
            &\geq \frac{n \cdot m + m}{2^{14}} \label{eq2:chain:4} \\ 
            &\geq  2^{-14} \cdot \left(\min(m, n^2) \cdot \min(m, n) + m\right). \notag
        \end{align}  
    \end{subequations}

    Inequality (\ref{eq2:chain:3}) follows from the assumption $n^2 \ge m$. Inequality (\ref{eq2:chain:4}) holds since $n \ge 1$ and, therefore, $\frac{n \cdot m}{2} \ge \frac{m}{2}$. 

    Both inequalities (\ref{triploid_relation}) and (\ref{triploid_second_relation}) automatically follow from (\ref{eq:mlown}).
\end{proof}

\begin{lemma} \label{mbign}
    For any integers $n, m \geq 4$ satisfying $n^2 < m$ there exists an oriented graph $G = (V, E)$ with $|V| = n$ and $|E| = m$ such that inequalities (\ref{triploid_relation}) and (\ref{triploid_second_relation}) hold and such that $G \cong T(n_1, n_2, n_3, n_0, e)$ for some integers $n_1, n_2, n_3, n_0, e$.
\end{lemma}

\begin{proof}
    Let us take \( G = T\left(\left\lfloor \frac{n}{4} \right\rfloor, \left\lfloor \frac{n}{2} \right\rfloor, \left\lfloor \frac{n}{4} \right\rfloor, n - 2\cdot  \left\lfloor \frac{n}{4} \right\rfloor- \left\lfloor \frac{n}{2} \right\rfloor, m \right)\). It is well-defined since
    $$m > n^2 \ge \left\lfloor \frac{n}{4} \right\rfloor \cdot \left\lfloor \frac{n}{2} \right\rfloor \cdot 2.$$  
    Then 
    \begin{subequations}\label{eq3:chain}
      \makeatletter
      \renewcommand{\theequation}{\arabic{parentequation}.\arabic{equation}}
      \makeatother
      
        \begin{align}
        \L(G) = m - 2 \cdot \left\lfloor \frac{n}{4} \right\rfloor \cdot \left\lfloor \frac{n}{2} \right\rfloor & \geq m - 2\cdot \frac{n}{4} \cdot \frac{n}{2} \notag \\ 
        & = m - \frac{n^2}{4} \notag \\
        & \geq m - \frac{n^2}{2} = \frac{2m - n^2}{2} \notag \\
        & \geq \frac{m}{2}. \label{eq3:chain:1}
        \end{align}
    \end{subequations}
    Note that inequality (\ref{eq3:chain:1}) holds since we have $m > n^2$ by assumption. Also, we have

    \begin{subequations}\label{eq5:chain}
      \makeatletter
      \renewcommand{\theequation}{\arabic{parentequation}.\arabic{equation}}
      \makeatother
      
        \begin{align}
        \Rh(G) 
        &\geq \left\lfloor \frac{n}{4} \right\rfloor^2 \cdot \left\lfloor \frac{\left\lfloor \frac{n}{2} \right\rfloor}{2} \right\rfloor \notag \\ 
        &\geq \left( \frac{n}{8} \right) ^2 \cdot \frac{\left\lfloor \frac{n}{2} \right\rfloor - 1}{2} \label{qechain:3} \\ 
        &\geq \left( \frac{n}{8} \right) ^2 \cdot \frac{\left\lfloor \frac{n}{2} \right\rfloor}{4} \geq \left( \frac{n}{8} \right) ^2 \cdot \frac{ n - 1 }{8}   \notag \\ 
        & \geq \left( \frac{n}{8} \right) ^2 \cdot \frac{ n }{16}   \notag \\ 
        &= \frac{n^3}{1024} \notag \\
        & \geq 2^{-14} \cdot \left(\min(m, n^2) \cdot \min(m, n)\right). \label{eq5:chain:1}
        \end{align}
    \end{subequations}
    Condition (\ref{qechain:3}) holds because $n \geq 4$. inequality (\ref{eq5:chain:1}) holds, since $m > n^2$. This proves the inequality (\ref{triploid_second_relation}).

    Finally, let us show that inequality  (\ref{triploid_relation}) holds.

        \begin{subequations}\label{eq5:chain}
      \makeatletter
      \renewcommand{\theequation}{\arabic{parentequation}.\arabic{equation}}
      \makeatother
      
        \begin{align}
        \L(G) + \Rh(G) 
        &\geq \frac{m}{2} + 2^{-14} \cdot \left(\min(m, n^2) \cdot \min(m, n)\right) \notag \\
        & \geq 2^{-14} \cdot \left(\min(m, n^2) \cdot \min(m, n) + m\right).
        \end{align}
    \end{subequations}
\end{proof}

Now we are able to finish the proof of the main result of this section.

\begin{proof}[Proof of Theorem \ref{nu_ge}]
    Lemmas \ref{m16}, \ref{mlown}, \ref{mn} and \ref{mbign} prove inequalities (\ref{triploid_relation}) and (\ref{triploid_second_relation}) for all possible cases of positive $n$ and $m$. Therefore, the statement of Theorem \ref{nu_ge} is true with constant $C = 2^{-14}$.
\end{proof}

\end{document}